\documentclass{article}
\usepackage{amsfonts, mathrsfs, amsmath, amsthm}
\usepackage{enumitem}
\usepackage{appendix}
\usepackage[section]{algorithm}
\usepackage{algpseudocode}
\usepackage{mathtools}
\usepackage{color}

\newtheorem{theorem}{Theorem}[section]
\newtheorem{lemma}[theorem]{Lemma}
\newtheorem{proposition}[theorem]{Proposition}
\newtheorem{corollary}{Corollary}[section]
\newtheorem{assumption}{Assumption}[section]
\theoremstyle{definition}

\DeclareMathOperator{\dist}{dist}

\numberwithin{equation}{section}

\title{Weak approximation of stochastic differential equations with sticky boundary conditions }
\author{A. Sharma\thanks{Department of Mathematical Sciences, Chalmers University of Technology and University of Gothenburg, Sweden; akashs@chalmers.se}}

\date{}

\providecommand{\keywords}[1]
{
  \small	
  \noindent\textbf{\textit{Key words. }} #1
}

\providecommand{\msc}[1]
{
  \small	
  \noindent\textbf{\textit{AMS subject classifications.}} #1
}

\begin{document}

\maketitle
\begin{abstract}

Sticky diffusion models a Markovian particle experiencing reflection and temporary adhesion phenomena at the boundary. Numerous numerical schemes exist for approximating stopped or reflected stochastic differential equations (SDEs), but this is not the case for sticky SDEs. In this paper, we construct and analyze half-order and first-order numerical schemes for the weak approximation of stochastic differential equations with sticky boundary conditions. We present the algorithms in general setting such that they can  be used  to  solve general linear parabolic partial differential equations with second-order sticky boundary condition via the  probabilistic representations of their solutions. Since the sticky diffusion spends non-zero amount of time on boundary, it poses extra challenge in designing the schemes and obtaining their order of convergence. We support the theoretical results with numerical experiments.  
    
    \bigskip
 \keywords{sticky diffusion, second-order boundary condition, mesh-free Monte Carlo methods for parabolic partial differential equations, Feynman-Kac formula.}
    
    \vspace{8pt}

\msc{ 65C30, 60H35, 60H10, 37H10}  
\end{abstract}
\section{Introduction}

This paper investigates the numerical schemes to simulate sticky diffusion. The applications of sticky diffusion arise in biology where molecules can viscously diffuse near a cell membrane \cite{w19}; in epidemics where there is a constraint on concentration of the pathogen of being non-negative but can be sticky near zero \cite{w20}; in queue modeling, inventories process where sticky behavior arises as a limit of storage process \cite{w21}; in finance where SDEs with sticky boundary is a good candidate to model interest rates which can be sticky near zero \cite{w23,w22}.  As is also highlighted in \cite{nawaf_miranda20}, the modeling  of the dynamics of mesoscopic particles is effective with SDEs portraying sticky behavior and its numerical simulation can give fundamental physical insight helpful in studying common materials \cite{w26,w24,w25,w27, holmes2020simulating_stratification_sticky_particles}.

To write the sticky SDEs we need the following notation. We denote with $G \subset \mathbb{R}^{d}$ a bounded domain having sufficiently smooth boundary $\partial G$. Let $Q := (0,T)\times G$. Let $b : \bar{Q} \rightarrow \mathbb{R}^{d}$ and $\sigma : \bar{Q} \rightarrow \mathbb{R}^{d\times d}$.  Let $I_{G}(x)$ and $I_{\partial G}(z)$ denote the indicator functions for $x \in G$ and $z \in\partial G $, respectively. We denote $\nu(z)$ as the inward normal vector at $z \in \partial G$.  Let $\varrho(z)$ and $\mu(z)$ be the scalar valued functions defined on $\partial G$.
Then, sticky diffusion is governed by the following stochastic differential equations evolving on $\bar G$:
\begin{align}
    dX(s) &= I_{G}(X(s))b(s,X(s))ds + I_{G}(X(s))\sigma(s,X(s))dW(s) \nonumber \\ & \;\;\;\;\; + I_{\partial G}(X(s))\varrho(X(s))\nu(X(s))dL(s),\;\;\;s \geq t_{0}, \; X(t_{0}) = x, \label{alsde1} \\ 
    \mu(X(s))dL(s) &= I_{\partial G}(X(s))ds, \label{alsde2}
\end{align}    
where \vspace{-3pt}
\begin{enumerate}[label= (\roman*)]
 \item $W(s) = (W^{1}(s),\dots,W^{d}(s))^{\top}$ is $d-$dimensional Wiener process, defined on a filtered probability space $(\Omega,\mathscr{F},(\mathscr{F}_{s})_{s\geq 0},\mathbb{P})$, with bracket process $[W^{i},W^{j}](s) = \delta_{ij}s$, \vspace{-3pt}
    \item $L(s)$ is an $\mathcal{F}_{s}$ adapted non-decreasing process such that for all $t>0$, $L(t) = \int_{0}^{t}I_{\partial G}(X(s))dL(s)$, a.s. $L(s)$ is also known as \textit{local time of $X(s)$ on the boundary}. It increases only when $X(s)$ hits the boundary $\partial G$.
    
\end{enumerate}

Partial differential equations play a fundamental role in physics, biology, finance, and engineering. The expectation of functions of sticky diffusions can be associated with solutions to linear parabolic partial differential equations (PDEs) featuring sticky (second-order) boundary conditions which we now
 introduce below. To this end, we need the following notation. Let $a : \bar{Q} \rightarrow \mathbb{R}^{d\times d}$.  Let $g(t,x)$ and $c(t,x)$ be scalar valued function defined on $\bar{Q}$.  Let $\gamma(z)$ be a scalar valued function defined on $\partial G$. If we assume that $\gamma(z)$, $\varrho(z)$, $\mu(z) $ are sufficiently smooth, by Seeley's theorem, we can say that $\gamma(z)$, $\varrho(z)$ and $\mu(z)$ are restrictions of functions defined on whole of $\bar{G}$.  Also, $\varphi(x)$ is a scalar valued function defined on $\bar{G}$.
Consider the following parabolic partial differential equation:
\begin{align}
 \frac{\partial u}{\partial t}  +  \mathcal{A}u  + c(t,x)u + g(t,x)  &:= \frac{\partial u}{\partial t} + \frac{1}{2}\sum\limits_{i,j=1}^{d}a^{ij}(t,x)\frac{\partial^{2}u}{\partial x^{i}\partial x^{j}}  +  \sum\limits_{i=1}^{d}b^{i}(t,x)\frac{\partial u}{\partial x^{i}} \nonumber\\  & \;\;\; + c(t,x)u + g(t,x) = 0,\;\;\; (t,x) \in [0,T]\times \bar{G},  \label{wbp1}
\end{align}
with terminal condition
\begin{equation}
    u(T,x) = \varphi(x),\;\;\; x \in \bar{G},\label{wbp2}
\end{equation}
and sticky boundary condition
\begin{align}\label{wbp3}
     \mathscr{L}u : &= -\underbrace{\mu(z)\bigg(\frac{1}{2}\sum\limits_{i,j=1}^{d}a^{ij}(t,z)\frac{\partial^{2}u}{\partial z^{i}\partial z^{j}} + \sum\limits_{i=1}^{d}b^{i}(t,z)\frac{\partial u}{\partial z^{i}} \bigg)}_{\text{(I)}} + \underbrace{\varrho(z)\frac{\partial u}{\partial \nu}}_{\text{(II)}}  + \underbrace{\gamma(z)u}_{\text{(III)}}\nonumber \\ & 
    = \psi(t,z),\;\;\;\;\;\; (t,z) \in [0,T]\times\partial G,
\end{align}
where terms (I), (II), (III) represent  sticky, reflecting and absorbing phenomena, respectively. We recall some terminologies here. If (I) $\equiv 0 $ and $\varrho \geq \varrho_{0} > 0 $, then (\ref{wbp3}) reduces to Robin boundary condition. In addition, if (III) $\equiv 0$ then (\ref{wbp3}) becomes Neumann boundary condition. If (I) and (II) $\equiv 0$ then (\ref{wbp3}) is called as Dirichlet boundary condition provided $\gamma \neq 0$ for any $z \in \partial G$. We call $\mu$ as coefficient of stickiness.  The solution of (\ref{wbp1})-(\ref{wbp3}) can be represented via Feynman-Kac formula where the underlying Markov process is governed by SDEs \eqref{alsde1}-\eqref{alsde2}.  These equations describe diffusive particle dynamics in bounded domains where the boundary exhibits all three effects :  reflection, adhesion, and absorption.
Note that second-order differential operator $\mathcal{A}$ in the boundary condition should not be confused with second-order differential operator corresponding to tangential diffusion, hence no atlas machinery is required. 

 The computational method to simulate one dimensional sticky Brownian motion is discussed in \cite{nawaf_miranda20} and to approximate one dimensional sticky diffusion in \cite{w11}. For multi-dimensional case with $G$ being half-space i.e. $G= \{(x_1,\dots,x_d)\; ; \; x_d>0\}$ and functions $g $, $c$, $\varrho $, $\gamma$, $\psi$ being identically zero, first-order method is proposed in \cite{meier2023simulation}. This method is  based on continuous-time Markov chain whose transition probabilities are calculated via finite difference discretization of the generator of SDEs in half-space. This makes it computationally intensive since at each time step of each trajectory in Monte Carlo simulation one needs to estimate the transition probabilities by discretizing generator. This is not the case for random walk based probabilistic methods, which weakly approximate reflected, confined or stopped diffusions, and hence can be used to solve  Neumann (in general, Robin), specular or Dirichlet boundary value problems (e.g. see \cite{1, costantini_pachhiarotti_sartoretto98, 48,64,2,40, gobet2000weak_killed_diff, leimkuhler2023simplerandom, sharma2022random, leimkuhler2024numerical}).  The sticky diffusion case which provides probabilistic representation to the solution of (\ref{wbp1})-(\ref{wbp3}) is not considered in these references. 
 
 In this paper, we propose random walks for approximating SDEs (\ref{alsde1})-(\ref{alsde2}). 
 In contrast to instantaneous reflection, the sticky diffusion spends non-zero amount of time on the boundary (see (\ref{alsde2}) and (I) in (\ref{wbp3})) which is mirrored in the construction of first-order (see \textbf{Algorithm~\ref{sticky_euler_method}}) as well as half-order  (see \textbf{Algorithm~\ref{proj_euler_method_stk_diff_algo}}) numerical schemes. We present the  algorithms in the general setting for approximating linear parabolic PDEs with second-order sticky boundary conditions.

In Section~\ref{sticky_Sec_2}, we impose the assumptions needed for well-posedness and present the probabilistic representation of the sticky boundary value problem.  In Section~\ref{sticky_num_method_section}, we develop numerical methods to approximate the sticky SDEs in weak-sense. The algorithms (first-order as well as half-order) are presented in the general setting such that they can be employed to approximate sticky boundary value problem. In Section~\ref{sticky_conv_proof_sticky_euler}, we prove the convergence of first-order method which we call as  sticky Euler method. In Section~\ref{sticky_proj_euler_section}, we establish the convergence of half-order projected Euler method. In Section~\ref{sticky_sec_num_exp}, we present the findings of numerical experiments validating the theoretical results.

\section{Setting}\label{sticky_Sec_2}
In this paper, we employ the following functional spaces. Let $C^{\frac{p+\epsilon}{2},p+\epsilon}(\bar{Q})$  denote a Hölder space consisting of functions $u(t,x)$  whose partial derivatives 
$\frac{\partial^{i+|j|}u}{\partial t^{i} \partial x^{j_{1}}\cdots \partial x^{j_{d}}}$
(or 
$\frac{\partial^{|j|}u}{\partial x^{j_{1}}\cdots \partial x^{j_{d}}})$
satisfying $2i + |j| < p+ \epsilon$ (or $|j| < p+\epsilon$) are continuous on $\bar{Q}$ (or $\bar{G}$) with finite norm $|\cdot|_{Q}^{(p+\epsilon)}$ (or $|\cdot|_{G}^{(p+\epsilon)}$). Here, $i \in \mathbb{N}\cup \{0\}$, $p \in \mathbb{N}\cup \{0\}$, $0 <\epsilon < 1$, $j$ represents a multi-index, and $|\cdot|^{p+\epsilon}$ denotes the Hölder norm (see \cite[pp. 7-8]{3} for details). The notation $v(t,z) \in C^{p + \epsilon}(\bar{S})$ (or $v(z) \in C^{p+\epsilon}(\partial G)$)  will have the same meaning as explained above where $S:= (0,T) \times \partial G$.

 For $d$-dimensional vectors  $V_{j}$, $j =1,\dots,p$, 
we denote the $p$-th derivative of a smooth function $v(y)$ evaluated in
the directions  $V_{j}$ by $D^{p}v(y)[V_{1},\dots,V_{p}]$:
\begin{equation*}\label{notationequation}
  D^{p}v(y)[V_{1},\dots,V_{p}] =  \sum_{i_1,\dots,i_p=1}^{d}\frac{\partial^{p} }{\partial y^{i_{1}}\dots\partial y^{i_{p}}} v(y) \, \prod\limits_{j=1}^{p}V_{j}^{i_{j}}.
\end{equation*}
In particular, $Dv(x)[V] = (\nabla v \cdot V)$, where we denote the scaler product between two vectors $a, b \in \mathbb{R}^d$ as $(a \cdot b)$.
By $\dist(x, \partial G)$, we will denote the shortest distance of a point $x$ from the boundary $\partial G$.

\noindent \textbf{SDEs with boundary conditions.}
 To ensure the existence and uniqueness of solution of (\ref{alsde1})-(\ref{alsde2}), we need certain assumptions (see \cite{w3,w4,w2,w1,graham1988martingale}).
\begin{assumption}\label{boundary_assum} $ G$ is a bounded domain and $\partial G \in C^{4,\epsilon}$. 
\end{assumption}

\begin{assumption}\label{assump_para_coeff}
$b \in C^{1+\epsilon, 2+ \epsilon}(\bar{Q})$, $\sigma \in C^{1+\epsilon, 2+ \epsilon}(\bar{Q})$, $ \mu \in C^{ 2+ \epsilon}(\partial G) $ and  $ \varrho \in C^{ 2+ \epsilon}(\partial G) $.
\end{assumption} 

\begin{assumption}\label{sticky_assum} $ \varrho(z) > \varrho_0 > 0$ and $ \mu(z) > 0$.
\end{assumption}

The assumption $ \varrho(z) > 0$  is called as strong transversality condition as opposed to weak transversality $\inf\limits_{z \in \partial G}(\varrho(z) + \mu(z)) > 0 $. The existence and uniqueness will hold with weak transversality condition but for constructing numerical schemes we need the above assumption. 

\paragraph{Probabilistic representation of parabolic sticky boundary value problem.} We first impose assumptions with regards to boundary value problem (\ref{wbp1})-(\ref{wbp3}).
 
 \begin{assumption}\label{coeff_assum2}
 $c \in C^{1+\epsilon, 2+ \epsilon}(\bar{Q})$, $g \in C^{1+\epsilon, 2+ \epsilon}(\bar{Q})$, $ \gamma \in C^{1+\epsilon, 2+ \epsilon}(\bar{S})$, $ \psi \in C^{1+\epsilon, 2+ \epsilon}(\bar{S}) $ and $ \varphi \in C^{ 4+ \epsilon}(\partial G) $.
 \end{assumption}

\begin{assumption}\label{ellip_assum}
    We assume that a unique solution of sticky boundary value problem (\ref{wbp1})-(\ref{wbp3}) exists and belongs to $C^{4, \epsilon}(\bar{Q})$. 
\end{assumption}
These assumptions are standard in the literature on convergence proofs for random walk based PDEs solvers   \cite{mil_tretyakov_book, milstein2003simplest_dirichlet}.  The proof of existence and uniqueness of classical solution of elliptic PDEs with sticky boundary condition is provided in \cite{cattiaux1992} (assuming hypoellipticity of diffusion coefficient).

The probabilistic representations of solutions of  general linear Dirichlet and Robin boundary value problems are obtained in \cite{freidlin_book}. Below we obtain the probabilistic representation for sticky boundary value problem \eqref{wbp1}-\eqref{wbp3} following the same set of arguments. 
\begin{proposition}
    Let Assumptions~\ref{boundary_assum}-\ref{ellip_assum} hold. Then, the probabilistic representation of the boundary value problem \eqref{wbp1}-\eqref{wbp3} is given by
    \begin{align}
       u(t_0, x) =  \mathbb{E}\big( \varphi(X(s)) Y(s) + Z(s)\big),
    \end{align}
  where  $X(s)$ is from (\ref{alsde1})-(\ref{alsde2}) with $X(t_0) = x$, $a = \sigma \sigma^{\top}$, and  $Y(s)$ and $Z(s)$ are governed by
\begin{align} 
    &dY(s) = c(s,X(s))Y(s)ds + I_{\partial G}(X(s))\gamma(s,X(s))Y(s)dL(s),\;\;\; Y(t_{0}) = 1,\label{alsde3}\\
    &dZ(s) = g(s,X(s))Y(s)ds  - I_{\partial G}(X(s))\psi(s,X(s))Y(s)dL(s),\;\;\; Z(t_{0}) = 0.\label{alsde4}
\end{align}
\end{proposition}
\begin{proof}

Let $R(t) = \int_{t_0}^{t}c(s, X(s))ds$ and $P(t) = \int_{t_0}^{t}\gamma(s, X(s))d L(s)$. We apply Ito's formula on $u(t,x)e^{r + p}$ to get
\begin{align}
u(T,& X(T))e^{R(T) + P(T)} = u(t_0, x) + \int_{t_0}^{T}\Big(\frac{\partial u}{\partial s}(s, X(s)) + I_{ G}(X(s))\mathcal{A}  u(s, X(s)) \nonumber \\ & \;\;\; + c(s, X(s))u(s,X(s)) \Big)e^{R(s) + P(s)}ds   + \int_{t_0}^{T}e^{R(s) + P(s)}\big( \nabla u(s, X(s)) \cdot \sigma(s, X(s))dW(s)\big) \nonumber \\ & \;\;\; + \int_{t_0}^{T} \Big(\varrho(X(s))(\nu(X(s)) \cdot \nabla u(s, X(s)))   +  \gamma(s, X(s))u(s,X(s))\Big) e^{R(s) + P(s)} dL(s) \nonumber \\ & 
= u(t_0, x) + \int_{t_0}^{T}\Big(\frac{\partial u}{\partial s}(s, X(s)) + \mathcal{A}  u(s, X(s)) - I_{\partial G}(X(s)) \mathcal{A}u(s,X(s)) \nonumber \\ & \;\;\; + c(s, X(s))u(s,X(s)) \Big)e^{R(s) + P(s)}ds   + \int_{t_0}^{T}e^{R(s) + P(s)}\big(\nabla u(s, X(s)) \cdot \sigma(s, X(s))   dW(s)\big) \nonumber \\ & \;\;\; + \int_{t_0}^{T} \Big( \varrho(X(s))(\nu(X(s)) \cdot \nabla u(s, X(s)))   +  \gamma(s, X(s))u(s,X(s))\Big) e^{R(s) + P(s)} dL(s),
\end{align}
Using (\ref{alsde2}) and the fact that $G$ is a bounded domain, we can write
\begin{align*}
u(&T, X(T))e^{R(T) + P(T)} = 
u(t_0, x) + \int_{t_0}^{T}\Big(\frac{\partial u}{\partial s}(s, X(s)) + \mathcal{A}  u(s, X(s)) \nonumber \\ & \;\;\; + c(s, X(s))u(s,X(s)) \Big)e^{R(s) + P(s)}ds   + \int_{t_0}^{T}e^{R(s) + P(s)}\big(\nabla u(s, X(s)) \cdot \sigma(s, X(s))   dW(s)\big) \nonumber \\ & \;\;\; + \int_{t_0}^{T} \Big(- \mu(X(s))\mathcal{A}u(s,X(s)) + \varrho(X(s))(\nu(X(s)) \cdot \nabla u(s, X(s)))  \\  &  \;\;\;  +  \gamma(s, X(s))u(s,X(s))\Big) e^{R(s) + P(s)} dL(s),
\end{align*}
which on applying (\ref{wbp1})-(\ref{wbp3}) and taking expectation on both sides gives
\begin{align*}
 \mathbb{E}\varphi(X(T))Y(T) = u(t_{0}, x) + \mathbb{E}\Big(-\int_{t_{0}}^{T}g(s,X(s))Y(s)ds + \int_{t_{0}}^{T}\psi(s, X(s))Y(s)dL(s)\Big),
\end{align*}
where $Y(s) = e ^{R(s) + P(s)}$. Therefore, we have the following probabilistic representation:
\begin{equation}
    u(t_{0},x) = \mathbb{E}(\varphi(X(s))Y(s) + Z(s)),
\end{equation}
where $X(s)$ is from (\ref{alsde1})-(\ref{alsde2}), and $Y(s)$ and $Z(s)$ are from \eqref{alsde3}-\eqref{alsde4}.
\end{proof}

%%%%%%%%%%%%%%%%%%%%%%%%%%%%%%%%%%%%%%%%%%%%%%%%%%%%%%%%%%%%%%%%%%%%%%%%%%%%%%%%%%%%%%%%%%%%%%%%%%%%%%%%%%%%%%%%%%%%%%%%%%%%%%%%%%%%%%%%%%%%%%%%%%%%%%%%%%%%%%%%%%%%%%%%%%%%%%%%%%%%%%%%%%%%%%%%%%%%%%%%%%%%%%%%%%%%%%%%%%%%%%%%%%%%%%%%%%%%%%%%%%%%%%%%%%%%%%%%%%%%%%%%%%%%%%%%%%%%%%%%%%%%%%%%%%%%%%%%%%%%%%%%%%%%%%%%%%%

\section{Numerical Schemes}\label{sticky_num_method_section}
 We divide our discussion into two subsections. 
 In Subsection~\ref{subsection3.2}, we construct Markov chain for first-order approximation of sticky diffusion.  In Subsection~\ref{subsection3.1}, we introduce half-order method to simulate sticky diffusion with reflection. 
 
To be precise, the aim here is to construct a Markov chain $(t_{k}, X_{k}, Y_{k}, Z_{k})_{k > 0}$ approximating (\ref{alsde1})-(\ref{alsde2}) and (\ref{alsde3})-(\ref{alsde4}). For brevity, we will take $\varrho \equiv 1$, but the schemes can be generalized without any difficulty for general $\varrho$ provided Assumption~\ref{sticky_assum} holds.  Let $h \in (0, 1)$ be fixed. Let $\chi$ denote the number of steps which may be random. The Markov chain starts at $(t_{0}, X_{0}, Y_{0}, Z_{0})= (t_{0}, x, 1, 0)$.

%%%%%%%%%%%%%%%%%%%%%%%%%%%%%%%%%%%%%%%%%%%%%%%%%%%%%%%%%%%%%%%%%%%%%%%%%%%%%%%%%%%%%%%%%%%%%%%%%%%%%%%%%%%%%%

\subsection{Sticky Euler method} \label{subsection3.2}

 Consider $X_0 $ belonging to $G$. Given $(t_k, X_k, Y_k, Z_k)$, the Markov chain takes auxiliary step which we denote as $(t_{k+1}^{'}, X_{k+1}^{'}, Y_{k+1}^{'}, Z_{k+1}^{'})$ with updating rules: 
\begin{align}
X_{k+1}^{'} &= X_{k } + hb(t_k, X_{k}) + h^{1/2}\sigma(t_k, X_k)\xi_{k+1},\label{sticky_eulerscheme_x}\\ 
t_{k+1}^{'} &= t_{k} + h, \label{sticky_eulerscheme_t}
\end{align}
and
\begin{align}
Y_{k+1}^{'} &= Y_{k} + hc(t_{k},X_{k})Y_{k},\label{sticky_eulerscheme_y} \\
Z_{k+1}^{'} &= Z_{k} + hg(t_{k},X_{k})Y_{k},\label{sticky_eulerscheme_z} 
\end{align}
where $\xi_{k+1} := \{ \xi_{k+1}^{1}, \dots, \xi_{k+1}^{d} \}$, $k = 0,\dots, \chi-1$, and each component of $\xi_{k+1}$ is independent and identically distributed taking values $\pm 1$ each with probability $1/2$. 
Based on the auxiliary steps mentioned above, the following four cases are considered. The first case considers when no boundary cross-over occurs and rest three cases take into account different scenarios when the boundaries of the cylinder $Q$ are crossed by  $(t_{k+1}^{'}, X_{k+1}^{'})$.

\begin{description}
     \item[\textbf{Case I} $t^{'}_{k+1} < T$ and $X^{'}_{k+1} \in \bar{G}$. ]
     If $X_{k+1}^{'} \in \bar{G}$ and $t_{k+1}^{'} < T$, then the updating rules are simply as follows:  
\begin{align}
X_{k+1} &= X_{k+1}^{'}, \;\;\; 
t_{k+1} = t_{k+1}^{'}, \label{sticky_eqn_3.5}\\ 
Y_{k+1} &= Y_{k+1}^{'}, \;\;\;
Z_{k+1} = Z_{k+1}^{'}. \label{sticky_eqn_3.6}
\end{align}
 \item[\textbf{Case II} $t^{'}_{k+1} \geq T$ and $X^{'}_{k+1} \in \bar{G}$. ]
 However, if $X_{k+1}^{'} \in \bar{G}$ and $t_{k+1}^{'} \geq T$ then we assign $X_{\chi} = X_{k+1}^{'}$ and $ t_\chi = T$. Also, $Y_{\chi} $ and $Z_{\chi}$ are computed according to (\ref{sticky_eulerscheme_y}) and (\ref{sticky_eulerscheme_z}).

 \item[\textbf{Case III} $t^{'}_{k+1} < T$ and $X^{'}_{k+1} \in \bar{G}^c$. ]
  In this case and next case, we have $X_{k+1}^{'} \in \bar{G}^c$.  Thanks to the bounded random variables and the boundedness of coefficients $b$ and $\sigma$, it is not difficult to deduce that $\dist(X_{k+1}^{'}, \partial G) = \mathcal{O}(h^{1/2})$.
   We calculate projection of $X_{k +1}^{'}$ on $\partial G $ and calculate the distance $r_{k+1} := \dist(X_{k+1}, \partial G)$ by solving
\begin{align}
X_{k+1}^{\pi} &= X_{k +1} + r_{k+1}\nu(X_{k+1}^{\pi}),\label{sym_proj_x}
\end{align}
where $X_{k+1}^{\pi}$ denotes the projection of $X_{k+1}$ on $\partial G $.  Due to Assumption~\ref{boundary_assum}, there exists a unique projection of $X_{k+1}^{'}$ on $\partial G$, for sufficiently small $h$, satisfying \eqref{sym_proj_x} (see \cite{2}).  

To portray the sticky behavior, an update rule is required for the temporal variable which will depend on $r_{k+1}$ and hence the update will be of order $\mathcal{O}(h^{1/2})$. This creates an issue at the final step i.e. if $t_{k+1}^{'}$ combined with the required update cross temporal boundary $\{T\}$. 
Hence, a boundary correction, in time variable, is needed to obtain the optimal first-order global convergence. Therefore, this case is subdivided into two cases based on the following auxiliary step in time:
 \begin{align}\label{checking_t''}
     t^{''}_{k+1} = t_{k+1}^{'} + 2 r_{k+1}\mu(X^{\pi}_{k+1}).
 \end{align}
 The above auxiliary step imitates the sticky behavior of continuous-time sticky dynamics. 
\begin{description}
    \item[ \textbf{Case IIIa} $t^{''}_{k+1} < T $. ] Then 
    \begin{align}\label{sticky_case3a_eqn_1}
        t_{k+1} = t^{''}_{k+1}, \quad \quad
        X_{k+1} = X_{k+1}^{'} + 2 r_{k+1}\nu(X^{\pi}_{k+1}). 
    \end{align}

\begin{figure}[H]
    \centering
\includegraphics[width=0.8\linewidth]{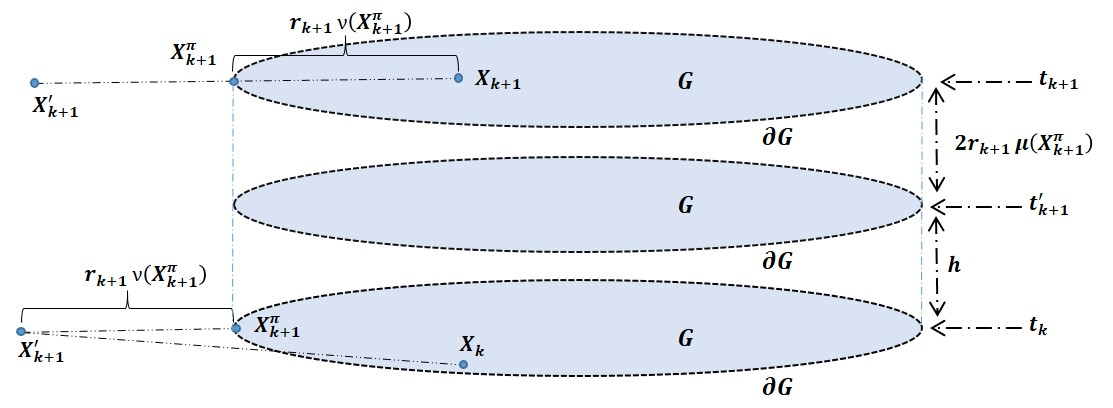}
    \caption{Schematic of one-step procedure (not to scale) when $X_{k+1}^{'} \in \bar{G}^c$. Here, we consider the cylinder $ G \times [t_k, t_{k+1}] $ depicting the movement of Markov chain from $(t_k, X_k)$ to $(t_{k+1}, X_{k+1})$. }
    \label{fig:placeholder}
\end{figure}    
 Denoting $t_{k+1/2}^{''} = t_{k+1}^{'} + r_{k+1} \mu(X^{\pi}_{k+1})$, the updating rules for $ Y_{k+1}$ and $Z_{k+1}$ are as follows:
\begin{align} 
Y_{k+1} &= Y_{k+1}^{'}  + 2r_{k+1}\gamma(t_{k+1/2}^{''}, X^{\pi}_{k})Y_{k} + 2r_{k+1}\mu(X_{k}^{\pi})c(t_{k+1/2}^{''}, X_{k}^{\pi})Y_{k}  
\nonumber  \\& \;\;\; 
+ 2r_{k+1}^{2}\mu(X_{k+1}^{\pi})\gamma^{2}(t_{k+1/2}^{''},X_{k+1}^{\pi})Y_{k} + 4r_{k+1}^{2}\mu(X_{k+1}^{\pi})\gamma(t_{k+1/2}^{''},X_{k}^{\pi}) c(t_{k+1/2}^{''},X_{k+1}^{\pi})Y_{k}  
\nonumber \\   &  \;\;\;   + 2r_{k+1}^{2}\mu^{2}(X_{k+1}^{\pi})c^{2}(t_{k+1/2}^{''},X_{k+1}^{\pi})Y_{k}  ,\label{sym_y} \\
Z_{k+1} &= Z_{k+1}^{'} -   2r_{k+1}\psi(t_{k+1/2}^{''}, X^{\pi}_{k+1})Y_{k} - 2 r_{k+1}^{2}\gamma(t_{k+1/2}^{''}, X^{\pi}_{k+1}) \psi(t_{k+1/2}^{''}, X^{\pi}_{k+1})Y_{k}
\nonumber \\& \;\;\;- 2r_{k+1}^{2}\mu( X^{\pi}_{k+1})c(t_{k+1/2}^{''}, X_{k+1}^{\pi}) \psi(t_{k+1/2}^{''}, X_{k+1}^{\pi})Y_k
\nonumber \\& \;\;\; + 2r_{k+1}\mu( X^{\pi}_{k+1})g(t_{k+1/2}^{''}, X^{\pi}_{k+1})Y_{k} 
+ 2r_{k+1}^2\mu(X^{\pi}_{k+1})\gamma(t_{k+1/2}^{''}, X^{\pi}_{k+1})g(t_{k+1/2}^{''}, X^{\pi}_{k+1})Y_k  \nonumber  \\& \;\;\;  +  2r_{k+1}^{2}(\mu( X^{\pi}_{k+1}))^{2}c(t_{k+1/2}^{''}, X^{\pi}_{k+1})g(t_{k+1/2}^{''}, X^{\pi}_{k+1})Y_k.\label{sym_z}
\end{align}
    \item[ \textbf{Case IIIb} $t^{''}_{k+1} \geq T $. ] Then 
    \begin{align}
        t_{\chi} &= T, \quad \quad 
        X_{\chi} = X_{k+1}^{'} + 2 r_{k+1}\nu(X^{\pi}_{k+1}),  \label{sticky_eqn_3.13}\\
        Y_{\chi} &=   Y_{k+1}^{'} + 2r_{k+1}\gamma(T, X^{\pi}_{k+1})Y_k -2 p_{k+1}\mu(X^{\pi}_{k+1}) c(T, X^{\pi}_{k+1}) Y_{k},\label{sticky_eqn_3.14}  \\
        Z_{\chi}  &= Z_{k+1}^{'} - 2 \mu(X^{\pi}_{k+1}) ( r_{k+1} - p_{k+1})\mathcal{A}\varphi(X^{\pi}_{k+1})Y_k \nonumber  \\  & \quad -  2 p_{k+1}\mu(X^{\pi}_{k+1}) g(T, X^{\pi}_{k+1})Y_{k} - 2r_{k+1}\psi(T, X^{\pi}_{k+1})Y_k, \label{sticky_eqn_3.15}
    \end{align}
    where 
    \begin{align}
        p_{k+1} = (T-t^{'}_{k+1})/(2\mu(X_{k+1}^{\pi})). \label{sticky_eqn_3.16}
    \end{align}
\end{description}

 \item[\textbf{Case IV} $t^{'}_{k+1} \geq T$ and $X^{'}_{k+1} \in \bar{G}^c$. ]

The transition rule for the subsequent step of the Markov chain is given by:
\begin{align}
    t_{\chi} &= T, \quad \quad 
    X_{\chi} = X^{'}_{k+1} + 2 r_{k+1}\nu(X^{\pi}_{k+1}), \label{stikcy_euler_eqn_3.16} \\ 
    Y_{\chi} &= Y_{k+1}^{'} + 2 r_{k+1}\gamma(T, X^{\pi}_{k+1})Y_k + 2 r_{k+1}^2\gamma^2(T, X^{\pi}_{k+1})Y_k, \label{sticky_euler_eqn_3.17}\\
    Z_{\chi} &= Z_{k+1}^{'} +[-2r_{k+1}\mu(X^{\pi}_{k+1})\mathcal{A}\varphi(X^{\pi}_{k+1})-2r^{2}_{k+1}\gamma(T, X^{\pi}_{k+1})\mu(X^{\pi}_{k+1})\mathcal{A}\varphi(X^{\pi}_{k+1}) \nonumber \\  &  \;\;\;\; - 2r_{k+1}\psi(T, X^{\pi}_{k+1})  - 2r_{k+1}\gamma(T, X^{\pi}_{k+1})\psi(T, X^{\pi}_{k+1})]Y_k. \label{sticky_euler_eqn_3.18}
\end{align}

\end{description}

%\@algocf@start

\setlength{\abovecaptionskip}{5pt}

\begin{algorithm}[htbp]
%\begin{center}
%\begin{minipage}{0.8\textwidth}
\small  % Smaller font
\setlength{\baselineskip}{0.9\baselineskip}
\caption{Sticky Euler method}\label{sticky_euler_method}
\begin{algorithmic}
\Require $X_0 = x$,  $t_0$, $T$, $h$
\Ensure $X_0 \in G$, $k = 0$, $ t = t_0$, $Y_0 = 1$ and $Z_0 = 0$
\State running = \textbf{True}
\While{running}
\State Simulate $ \xi_{k+1} $ 
and compute $X^{'}_{k+1}$, $t_{k+1}^{'}$,  $Y_{k+1}^{'}$ and $Z_{k+1}^{'}$ according to \eqref{sticky_eulerscheme_x}-\eqref{sticky_eulerscheme_z}
\If{$t^{'}_{k+1} < T$ and $X_{k+1}^{'} \in \bar{G}$ }
    \State  $X_{k+1} = X^{'}_{k+1}$, $t_{k+1} = t_{k+1}^{'}$,   $Y_{k+1} =Y_{k+1}^{'}$, $Z_{k+1} = Z_{k+1}^{'}$.
    \State $k \gets k+1$.
\ElsIf{$t^{'}_{k+1} \geq T$ and $X_{k+1}^{'} \in \bar{G}$}
\State $\chi \gets k+1$
\State $t_\chi = T$, $X_{\chi} = X^{'}_{k+1}$, $Y_{\chi } = Y_{k+1}^{'}$, $Z_{\chi} = Z_{k+1}^{'}$. 
\State running = \textbf{False}
\ElsIf{ $t_{k+1}^{'} < T$ and $X_{k+1}^{'} \in \bar{G}^{c}$ }
    \State (i) Find $r_{k+1}$ and $X_{k+1}^{\pi}$ using (\ref{sym_proj_x})
    \State (ii) Take checking step $ t^{''}_{k+1}$ as in (\ref{checking_t''})
     \If{$ t^{''}_{k+1} < T$}
    \State Calculate $X_{k+1}$, $t_{k+1}$, $Y_{k+1}$, $Z_{k+1}$,  according to (\ref{sticky_case3a_eqn_1})-(\ref{sym_z})

    \Else
    \State $\chi \gets k+1$
      \State Calculate $t_{\chi}$, $X_{\chi}$, $Y_{\chi}$, $Z_{\chi}$ according to (\ref{sticky_eqn_3.13})-(\ref{sticky_eqn_3.15})
      \State running = \textbf{False}      
    \EndIf     
    \ElsIf{$t_{k+1}^{'} \geq T$ and $X_{k+1}^{'} \in \bar{G}^{c}$}
    \State $\chi \gets k+1$
    \State Compute $t_\chi$, $X_{\chi}$, $Y_{\chi}$ and $Z_{\chi}$ according to \eqref{stikcy_euler_eqn_3.16}-\eqref{sticky_euler_eqn_3.18} 
    \State running = \textbf{False}
\EndIf
\EndWhile  
\State \textbf{Output:} $ \chi, X_\chi, Y_\chi, Z_\chi$ 
\end{algorithmic}

\end{algorithm}
\setlength{\belowcaptionskip}{5pt}

 Straightforward symmetrized procedure along the boundary $\partial G$ does not work for approximation of sticky diffusion to obtain first-order of weak convergence as it works for instantaneously reflected diffusion (see \cite{2, leimkuhler2023simplerandom}). To deal with it, we introduce symmetry in spatial and temporal variables. In spatial variable, this symmetry is around the boundary point, i.e. $X_{k+1}^{\pi}$ and  in temporal variable around a fictitious point $t_{k+1/2}^{''}$. The approximating chain spends $\mathcal{O}(h^{1/2})$ (to be precise $2\mu(X_{k+1}^{\pi}) r_{k+1}$) amount of time on the boundary when $X_{k+1}^{'} \in \bar{G}^{c}$. This implies that number of steps are random and a correction procedure is needed at the final time to maintain first-order of convergence (see \textbf{Case IIIb} and \textbf{Case IV} above). This is in contrast to earlier symmetrized schemes \cite{2} (based on Gaussian random variables) and \cite{leimkuhler2023simplerandom} (based on bounded random variables) for instantaneously reflecting diffusions.  The updating rules for $Y_{k}$ and $Z_{k}$ are  derived in the manner which  ensures optimal order of convergence.  This results in a completely different numerical scheme for sticky diffusion, i.e. Algorithm~\ref{sticky_euler_method} as compared to Algorithm~2 in \cite{leimkuhler2023simplerandom} for reflected diffusion.

\subsection{Projected Euler method for sticky diffusion}\label{subsection3.1}
 We take $X_0 $ belonging to $G$. Given $X_k \in G$ and $t_k < T$, the update rules are
\begin{align}
X_{k+1} &= X_{k } + hb(t_k, X_{k}) + h^{1/2}\sigma(t_k, X_k)\xi_{k+1}, \;\;\; t_{k+1} = t_{k} + h,\label{stickyproj_eulerscheme_x}\\ 
 %\label{sticky_eulerscheme_t} \\
Y_{k+1} &= Y_{k} + hc(t_{k},X_{k})Y_{k},\label{stickyproj_eulerscheme_y} \\
Z_{k+1} &= Z_{k} + hg(t_{k},X_{k})Y_{k},\label{stickyProj_eulerscheme_z} 
\end{align}
where $ \xi_{k+1} = (\xi_{k+1}^{1}, \dots, \xi_{k+1}^{d})$ and all $\xi_{k+1}^{i}$, $i=1,\dots,d$, $k = 1, \dots, \chi-1$, are independent and take values $\pm 1$ each with probability $1/2$.

Let $t_k < T$. If $X_{k} \in \bar{G}^{c}$, it implies in the previous step Markov chain moved out of $\bar{G}$.  We first calculate $r_k = \dist(X_k, \partial G)$. Given $X_k \in G^{c}$, the Markov chain takes steps as follows:
\begin{align}
X_{k+1} &= X_{k } + r_{k}\nu(X_{k+1}),\label{proj_x}\\ 
h_{k} &=  r_{k}\mu(X_{k+1}), \;\;\;\; 
t_{k+1} = t_{k} + h_{k}, \label{proj_t}\\
Y_{k+1} &= Y_{k} + r_{k}\gamma(t_{k},X_{k+1})Y_{k} + r_{k}\mu(X_{k+1})c(t_{k}, X_{k+1})Y_{k},\label{proj_y} \\
Z_{k+1} &= Z_{k} - r_{k}\psi(t_{k},X_{k+1})Y_{k} + r_{k}\mu(X_{k+1})g(t_{k}, X_{k+1})Y_{k},\label{proj_z}
\end{align}
where  $X_{k + 1}$ is the unique projection of $ X_k$ on $\partial G$ and $r_k =\mathcal{O}(h^{1/2})$ for all $k = 1,\dots,\chi-1$. The algorithm is written in more compact manner as Algorithm~\ref{proj_euler_method_stk_diff_algo}.

 The rate obtained for Algorithm~\ref{proj_euler_method_stk_diff_algo} is $1/2$  with average total number of steps of order $\mathcal{O}(h^{-1})$. Projection schemes have been proposed for simulating reflected and stopped diffusions (see \cite{costantini_pachhiarotti_sartoretto98} for half-order weak-method to simulate reflected diffusion with $\mu \equiv 0$, $g\equiv 0$ and $\varrho = 1$ ). Note that when $X_{k} \in \bar{G}^{c}$, $h_{k}$ is $\mathcal{O}(h^{1/2})$. This implies that there is non-zero probability that $t_{\chi} > T$, where $\chi$ denote number of step. This induces an error of $\mathcal{O}(h^{1/2})$ at the final step which does not disturb the rate of convergence.

\begin{algorithm}
\caption{Projected Euler method for sticky diffusion}\label{proj_euler_method_stk_diff_algo}
\small  % Smaller font
\setlength{\baselineskip}{0.9\baselineskip}
\begin{algorithmic}
\Require $X_0 = x$, $t_0$, $T$, $h$
\Ensure $X_0 \in G$, $k = 0$, $ t = t_0$,  $Y_0 = 1$ and $Z_0 = 0$

\While{$t \leq T$}
\State Simulate $ \xi_{k+1} $ 
\If{$X_k \in G$ }
   \State  Calculate $X_{k+1}$, $t_{k+1}$, $Y_{k+1}$ and $Z_{k+1}$  according to (\ref{sticky_eulerscheme_x})-(\ref{sticky_eulerscheme_t})
      %\Comment{This is a comment}
\Else
    \State (i) Find $r_k$ and $X_{k+1}$ using (\ref{proj_x})
    \State (ii) Calculate $t_{k+1}$, $Y_{k+1}$, $Z_{k+1}$ according to (\ref{proj_t})-(\ref{proj_z})
\EndIf
\State $t \gets t_{k+1}$
\State $k \gets k+1$
\EndWhile
\end{algorithmic}
\end{algorithm}

\section{Convergence of sticky Euler scheme} \label{sticky_conv_proof_sticky_euler}

\begin{theorem}\label{thrm_sticky_euler_conv}
Under Assumptions~\ref{boundary_assum}-\ref{ellip_assum}, the following inequality is satisfied for \textbf{Algorithm~\ref{sticky_euler_method}}:
\begin{align}
|\mathbb{E}( \varphi(X_{\chi})Y_{\chi} + Z_{\chi}) - u(t_{0}, X_{0})| \leq Ch, 
\end{align}
where $ u(t,x) $ is solution of (\ref{wbp1})-(\ref{wbp3}), and $C $ is a positive constant independent of $h$. 
\end{theorem}

We consider the one-step error analysis in Subsection~\ref{sticky_subsec_one_step_lemmas} and finally move to the proof of  Theorem~\ref{thrm_sticky_euler_conv} in Subsection~\ref{sticky_subsec_conv_proof}.

\subsection{One-step lemmas}\label{sticky_subsec_one_step_lemmas}
In this subsection, we analyze one-step error of Algorithm~\ref{sticky_euler_method} case by case. Lemma~\ref{sticky_lemma_4.2} considers Case~II, Lemma~\ref{sticky_lemma_4.3} and Lemma~\ref{sticky_lemma_4.4} examine Case~III and finally, Case~IV is investigated in Lemma~\ref{sticky_lemma_4.5}. Among these lemmas, the most crucial lemma is Lemma~\ref{sticky_lemma_4.3} for the global convergence theorem. For Case~I, we cite the standard result of one-step error of weak Euler scheme. 

\begin{lemma}\label{sticky_lemma_4.2}
    Let Assumptions~\ref{boundary_assum}-\ref{ellip_assum} hold, then we have the following one-step approximation:
\begin{align*}
|u_{k+1}Y_{k+1} + Z_{k+1} - u_{k+1}^{'}Y_{k+1}^{'} - Z_{k+1}^{'}   |I_{\bar{G}}(X_{k+1}^{'})I(t^{'}_{k+1} > T) \leq ChY_{k},\;\;\;\;\;a.s. 
\end{align*}
where $C > 0$ is independent of $h$. 
\end{lemma}
\begin{proof}
    In this case, $X_{k+1} = X_{k+1}^{'}$, $Y_{k+1} = Y_{k+1}^{'}$ and $Z_{k+1} = Z_{k+1}^{'}$. Using Taylor's approximation, we have

\begin{align}
    (u(T, X_{k+1}) - u(t_{k+1}^{'}, X_{k+1})) Y_{k+1} =  -(t_{k+1}^{'} - T)\frac{\partial u}{\partial t}(t_{k+1}^{'} + \epsilon (T-t_{k+1}^{'}))Y_{k+1},
\end{align}
where $\epsilon \in (0,1)$. Therefore, we get
$
     (u(T, X_{k+1}) - u(t_{k+1}^{'}, X_{k+1})) Y_{k+1} = \mathcal{O}(h)Y_{k},
$
where $\mathcal{O}(h)$ does not depend on $Y_{k}$. This completes the proof. 
\end{proof}

 We introduce the following notation which we will use in rest of the proofs: $u_{k+1} = u(t_{k+1},X_{k+1})$, $u_{k}= u(t_{k},X_{k})$, $u_{k+1}^{\pi}=u(t_{k+1/2}^{''},X_{k+1}^{\pi})$, $u^{\prime}_{k+1} = u(t_{k+1}^{'},X_{k+1}^{'})$, $a_{k} =a(t_{k}, X_{k})$, $b_{k}= b(t_{k},X_{k})$, $c_{k}= c(t_{k},X_{k})$, $g_{k}= g(t_{k},X_{k})$, $\psi_{k+1}^{\pi} = \psi(t_{k+1/2}^{''}, X_{k+1}^{\pi})$, $\gamma_{k+1}^{\pi}= \gamma(t_{k+1/2}^{''},X_{k+1}^{\pi})$, $\nu_{k+1}^{\pi} = \nu(X_{k+1}^{\pi})$, $ \varphi^{\pi}_{k+1} = \varphi(X_{k+1}^{\pi}) $ and $\sigma_{k} = \sigma(t_{k}, X_{k})$. 
  Under Assumption~\ref{boundary_assum}, the solution $u(t,x) \in C^{2,4}(\bar{Q})$ can be extended to  a function $u(t,x) \in C^{2,4}(\bar{Q}\cup\bar{Q}_{-r})$ (see \cite[Proposition 1.17]{19}). We extend $\psi \in C^{2}(\bar{G})$ in the similar manner. These extensions will be used in proofs. Note that $u(t,x)$ and its derivatives are uniformly bounded for $(t,x)\in \bar{Q} \cup \bar{Q}_{-r}$.

For the sake of convenience, let us denote $(t_{k+1}^{'}, X_{k+1}^{'})$, $(t_{k+1}, X_{k+1})$ and $(t_{k+1/2}^{''}, X_{k+1}^{\pi})$ as $P_{k+1}^{'}$, $P_{k+1}$ and $P_{k+1}^{\pi}$, where $t_{k+ 1/2}^{''} =t_{k+1}^{'} + h_{k+1} =  t_{k+1}^{'} + \mu(X^{\pi}_{k+1})r_{k+1}$. Also, $\Delta P_{k+1} := (h_{k+1}, r_{k+1}\nu_{k+1}^{\pi})$ with $h_{k+1} = \mu^{\pi}_{k+1}r_{k+1}$. It is clear that $P_{k+1} = P_{k+1}^{\pi} + \Delta P_{k+1}$ and $ P_{k+1}^{\pi} = P_{k+1}^{'} + \Delta P_{k+1}$. To avoid any confusion in notations we emphasize $u_{k+1}^{\pi} = u(P_{k+1}^{\pi})  = u(t_{k+1/2}^{''}, X_{k+1}^{\pi})$.
\begin{lemma}\label{sticky_lemma_4.3}
Let Assumptions~\ref{boundary_assum}-\ref{ellip_assum} hold, then we have the following one-step approximation:
\begin{align*}
|u_{k+1}Y_{k+1} + Z_{k+1} - u_{k+1}^{'}Y_{k+1}^{'} - Z_{k+1}^{'}   |I_{\bar{G}^{c}}(X_{k+1}^{'})I(t^{'}_{k+1} \leq T)I(t^{''}_{k+1} \leq T) \leq Chr_{k+1}Y_{k},\;\; \;a.s.
\end{align*}
where $C > 0$ is independent of $h$.
\end{lemma}

\begin{proof}

 Using Taylor's expansion, we have
\begin{align}
    u_{k+1} = u(P_{k+1}) &= u^{\pi}_{k+1} + Du_{k+1}^{\pi}[\Delta P_{k+1}] + \frac{1}{2}D^{2}u_{k+1}^{\pi}[\Delta P_{k+1},\Delta P_{k+1} ] \nonumber \\ & \;\;\; +\frac{1}{6} D^{3}u(P_{k+1} + \epsilon_{1}\Delta P_{k+1})[\Delta P_{k+1},\Delta P_{k+1}, \Delta P_{k+1} ],\label{wbpeq4.9}\\ 
    u_{k+1}^{'} = u(P_{k+1}^{'}) &= u_{k+1}^{\pi} - Du_{k+1}^{\pi}[\Delta P_{k+1}]+  \frac{1}{2}D^{2}u_{k+1}^{\pi}[\Delta P_{k+1},\Delta P_{k+1} ] \nonumber \\ & \;\;\; - \frac{1}{6}D^{3}u(P_{k+1} + \epsilon_{2}\Delta P_{k+1})[\Delta P_{k+1},\Delta P_{k+1}, \Delta P_{k+1} ], \label{wbpeq4.10}
\end{align}
where $\epsilon_{1}, \epsilon_{2} \in (0,1)$ and $D = (\partial/\partial t, \nabla)$. Subtracting (\ref{wbpeq4.10}) from (\ref{wbpeq4.9}), we get
\begin{align}
u_{k+1} - u_{k+1}^{'} &= 2 Du(P_{k+1}^{\pi})[\Delta P_{k+1}]+\frac{1}{6} D^{3}u(P_{k+1} + \epsilon_{1}\Delta P_{k+1})[\Delta P_{k+1},\Delta P_{k+1}, \Delta P_{k+1}] \nonumber \\ & \;\;\; + \frac{1}{6}D^{3}u(P_{k+1} + \epsilon_{2}\Delta P_{k+1})[\Delta P_{k+1},\Delta P_{k+1}, \Delta P_{k+1} ].  
\end{align}
Recall $h_{k+1} = \mu_{k+1}^{\pi} r_{k+1} $. Using (\ref{sym_y}) and \eqref{wbp1} with above expressions, we get
\begin{align}
(u_{k+1} - u_{k+1}^{'})Y_{k+1} &= \Big(2 h_{k+1}\frac{\partial u_{k+1}^{\pi}}{\partial t} + 2r_{k+1}(\nu_{k+1}^{\pi} \cdot \nabla u_{k+1}^{\pi}) + \mathcal{O}(r_{k+1}^{3})\Big) \Big(Y_{k}+ hc_kY_k \nonumber \\ &  \;\;\; + 2r_{k+1}\gamma_{k+1}^{\pi}Y_{k}   + 2r_{k+1}\mu_{k+1}^{\pi}c^{\pi}_{k+1}Y_{k} + 2r_{k+1}^{2}(\gamma_{k+1}^{\pi})^{2}Y_{k}  \nonumber  \\& \;\;\; + 4r_{k+1}^{2}\mu_{k+1}^{\pi}\gamma_{k+1}^{\pi} c_{k+1}^{\pi}Y_{k} + 2r_{k+1}^{2}(\mu_{k+1}^{\pi})^{2}(c_{k+1}^{\pi})^{2}Y_{k} \Big)
\nonumber \\ &  
= 2h_{k+1}\frac{\partial u_{k+1}^{\pi}}{\partial t}Y_{k} + 2r_{k+1}(\nu_{k+1}^{\pi} \cdot \nabla u_{k+1}^{\pi})Y_{k} + 4h_{k+1}r_{k+1}\gamma_{k+1}^{\pi}\frac{\partial u_{k+1}^{\pi}}{\partial t}Y_{k} 
\nonumber \\ & \;\;\;
+ 4r_{k+1}^{2}\gamma_{k+1}^{\pi} (\nu_{k+1}^{\pi} \cdot \nabla  u_{k+1}^{\pi}) Y_{k}   + 4h_{k+1}r_{k+1}\mu_{k+1}^{\pi}c_{k+1}^{\pi}\frac{\partial u_{k+1}^{\pi}}{\partial t}Y_{k} 
\nonumber \\ & \;\;\;
+ 4r_{k+1}^{2}\mu_{k+1}^{\pi}c_{k+1}^{\pi}(\nu_{k+1}^{\pi}\cdot \nabla u_{k+1}^{\pi})Y_{k} +\mathcal{O}(r_{k+1}^{3}Y_{k})
\nonumber  \\ & =
  \Big(-2h_{k+1}\mathcal{A}u_{k+1}^{\pi}  + 2r_{k+1}(\nu_{k+1}^{\pi}\cdot \nabla u_{k+1}^{\pi}) - 4h_{k+1}r_{k+1}\gamma_{k+1}^{\pi}\mathcal{A}u_{k+1}^{\pi} \nonumber  \\&  \;\;\; +  4r_{k+1}^{2}\gamma_{k+1}^{\pi} (\nu_{k+1}^{\pi} \cdot \nabla u_{k+1}^{\pi}) - 2h_{k+1}c_{k+1}^{\pi}u_{k+1}^{\pi} - 2h_{k+1}g_{k+1}^{\pi} \nonumber  \\&  \;\;\;  - 4h_{k+1}r_{k+1}\gamma_{k+1}^{\pi}c^{\pi}_{k+1}u_{k+1}^{\pi} - 4h_{k+1}r_{k+1}\gamma_{k+1}^{\pi}g^{\pi}_{k+1} - 4h_{k+1}r_{k+1}\mu_{k+1}^{\pi}c_{k+1}^{\pi} \mathcal{A}u_{k+1}^{\pi} \nonumber  \\&  \;\;\;  - 4h_{k+1}r_{k+1}\mu_{k+1}^{\pi}(c_{k+1}^{\pi})^{2}u_{k+1}^{\pi} - 4h_{k+1}r_{k+1}\mu_{k+1}^{\pi}c^{\pi}_{k+1}g^{\pi}_{k+1} \nonumber  \\& \;\;\; + 4r_{k+1}^{2}\mu_{k+1}^{\pi}c_{k+1}^{\pi}(\nu^{\pi}_{k+1}\cdot \nabla u_{k+1}^{\pi}) \Big)Y_{k}  
  + \mathcal{O}(r_{k}^{3}Y_{k}).\label{ref_eqn4.14}
\end{align}

Again applying (\ref{sym_y}), we ascertain
\begin{align}
u_{k+1}^{'}(Y_{k+1} &- Y_{k+1}^{'}) = \Big(u_{k+1}^{\pi} - h_{k+1}\frac{\partial u_{k+1}^{\pi}}{\partial t} - r_{k+1}(\nu_{k+1}^{\pi} \cdot\nabla u_{k+1}^{\pi}) + \mathcal{O}(r_{k+1}^{2})\Big) \Big( 2r_{k+1}\gamma_{k+1}^{\pi}Y_{k} \nonumber  \\  & 
\;\;\;\; + 2r_{k+1}\mu_{k+1}^{\pi}c^{\pi}_{k+1}Y_{k}  + 2r_{k+1}^{2}(\gamma_{k+1}^{\pi})^{2}Y_{k}  \nonumber  \\& \;\;\;\; + 4r_{k+1}^{2}\mu_{k+1}^{\pi}\gamma_{k+1}^{\pi} c_{k+1}^{\pi}Y_{k} + 2r_{k+1}^{2}(\mu_{k+1}^{\pi})^{2}(c_{k+1}^{\pi})^{2}Y_{k} \Big) \nonumber \\&  = 
\Big(u_{k+1}^{\pi} + h_{k+1} \mathcal{A}u_{k+1}^{\pi} + h_{k+1}c_{k+1}^{\pi}u_{k+1}^{\pi} + h_{k+1}g_{k+1}^{\pi} - r_{k+1}(\nu_{k+1}^{\pi} \cdot\nabla u_{k+1}^{\pi}) + \mathcal{O}(r_{k+1}^{2})\Big)
\nonumber \\  &   \;\;\; \times \Big( 2r_{k+1}\gamma_{k+1}^{\pi}Y_{k}  +2r_{k+1}\mu_{k+1}^{\pi}c_{k+1}^{\pi}Y_{k}   + 2r_{k+1}^{2}(\gamma_{k+1}^{\pi})^{2}Y_{k} \nonumber  \\& \;\;\; +  4r_{k+1}^{2}\mu_{k+1}^{\pi}\gamma_{k+1}^{\pi} c_{k+1}^{\pi}Y_{k} + 2r_{k+1}^{2}(\mu_{k+1}^{\pi})^{2}(c_{k+1}^{\pi})^{2}Y_{k} \Big) \nonumber \\
& = \Big( 2r_{k+1}\big[\gamma_{k+1}^{\pi}u_{k+1}^{\pi} + \mu^{\pi}_{k+1}c^{\pi}_{k+1}u^{\pi}_{k+1}\big] + 2 h_{k+1}r_{k+1}\big[\gamma_{k+1}^{\pi} \mathcal{A}u_{k+1}^{\pi} + \gamma_{k+1}^{\pi}c_{k+1}^{\pi}u_{k+1}^{\pi}  
\nonumber  \\ & \;\;\;\; + \gamma_{k+1}^{\pi}g_{k+1}^{\pi}   +\mu_{k+1}^{\pi}c_{k+1}^{\pi}\mathcal{A}u_{k+1}^{\pi}+\mu_{k+1}^{\pi}(c_{k+1}^{\pi})^{2}u_{k+1}^{\pi} 
+\mu_{k+1}^{\pi} c^{\pi}_{k+1} g_{k+1}^{\pi}\big] \nonumber 
\\  & \;\;\;\; 
+ 2r_{k+1}^{2}\big[ - \gamma_{k+1}^{\pi}(\nu_{k+1}^{\pi}\cdot \nabla  u_{k+1}^{\pi}) - \mu_{k+1}^{\pi}c_{k+1}^{\pi}(\nu_{k+1}^{\pi} \cdot \nabla u_{k+1}^{\pi}) +  (\gamma_{k+1}^{\pi})^{2}u_{k+1}^{\pi}  \nonumber  \\ & \;\;\;   + 2\mu_{k+1}^{\pi}\gamma_{k+1}^{\pi}c_{k+1}^{\pi} u_{k+1}^{\pi}  
 + (\mu_{k+1}^{\pi})^{2}(c_{k+1}^{\pi})^{2}u_{k+1}^{\pi}\big] \Big)Y_{k} + \mathcal{O}(r_{k+1}^{3}Y_{k}). \label{ref_eqn4.15}
\end{align}
Adding (\ref{ref_eqn4.14}) and (\ref{ref_eqn4.15}), and using the fact that $h_{k+1} = \mu_{k+1}^{\pi} r_{k+1}$, we get
\begin{align}
u_{k+1}Y_{k+1} &- u_{k+1}^{'}Y_{k+1}^{'} = 2r_{k+1}\Big(-\mu_{k+1}^{\pi}\mathcal{A}u_{k+1}^{\pi}+ (\nu_{k+1}^{\pi} \cdot \nabla u_{k+1}^{\pi}) + \gamma_{k+1}^{\pi}u_{k+1}^{\pi}\Big)Y_{k} 
\nonumber \\  &  \;\;\;\;  + 2r_{k+1}^{2}\gamma_{k+1}^{\pi}\Big(-\mu_{k+1}^{\pi}\mathcal{A}u_{k+1}^{\pi} + (\nu_{k+1}^{\pi}\cdot \nabla u_{k+1}^{\pi})  + \gamma_{k+1}^{\pi}u_{k+1}^{\pi}\Big)Y_{k} 
\nonumber  \\ & \;\;\; 
+ 2r_{k+1}^{2}\mu_{k+1}^{\pi}c_{k+1}^{\pi}\Big(- \mu_{k+1}^{\pi}\mathcal{A}u_{k+1}^{\pi} + (\nu_{k+1}^{\pi}\cdot \nabla u_{k+1}^{\pi}) + \gamma_{k+1}^{\pi}u_{k+1}^{\pi}\Big)Y_{k}  \nonumber \\ & \;\;\;  - 2r_{k+1}\mu_{k+1}^{\pi}g_{k+1}^{\pi} - 2(r_{k+1})^{2}\mu_{k+1}^{\pi}\gamma_{k+1}^{\pi}g_{k+1}^{\pi} - 2r_{k+1}^{2}(\mu_{k+1}^{\pi})^{2}c_{k+1}^{\pi}g_{k+1}^{\pi} \nonumber \\ & 
= \big(2r_{k+1}\psi_{k+1}^{\pi} + 2 r_{k+1}^{2}\gamma_{k+1}^{\pi} \psi_{k+1}^{\pi} + 2r_{k+1}^{2}\mu_{k+1}^{\pi}c_{k+1}^{\pi} \psi_{k+1}^{\pi} - 2r_{k+1}\mu_{k+1}^{\pi}g_{k+1}^{\pi} 
\nonumber 
\\ &  \;\;\; - 2r_{k+1}^{2}\mu_{k+1}^{\pi}\gamma_{k+1}^{\pi}g_{k+1}^{\pi} - 2r_{k+1}^{2}(\mu_{k+1}^{\pi})^{2}c_{k+1}^{\pi}g_{k+1}^{\pi}\big)Y_k + \mathcal{O}(Y_k r^{3}_{k+1}).
\end{align}
Using the relation \eqref{sym_z} i.e. 
\begin{align*}
    Z_{k+1} &= Z_{k+1}^{'} -   2r_{k+1}\psi^{\pi}_{k+1}Y_{k} - 2 r_{k+1}^{2}\gamma^{\pi}_{k+1} \psi^{\pi}_{k+1}Y_{k}
\nonumber \\& \;\;\;- 2r_{k+1}^{2}\mu^{\pi}_{k+1}c^{\pi}_{k+1} \psi^{\pi}_{k+1} + 2r_{k+1}\mu^{\pi}_{k+1}g^{\pi}_{k+1}Y_{k} 
\nonumber \\& \;\;\;+ 2r_{k+1}^2\mu^{\pi}_{k+1}\gamma^{\pi}_{k+1}g^{\pi}_{k+1}    +  2r_{k+1}^{2}(\mu^{\pi}_{k+1})^{2}c^{\pi}_{k+1}g^{\pi}_{k+1},
\end{align*}
we get the desired result. 
\end{proof}

The following lemma considers the Case IIIb discussed in Subsection~\ref{subsection3.2}. 

\begin{lemma}\label{sticky_lemma_4.4}
Let Assumptions~\ref{boundary_assum}-\ref{ellip_assum} hold, then we have the following one-step approximation:
\begin{align*}
|u_{k+1}Y_{k+1} + Z_{k+1} - u_{k+1}^{'}Y_{k+1}^{'} - Z_{k+1}^{'}   |I_{\bar{G}^{c}}(X_{k+1}^{'})I(t^{''}_{k+1} > T) \leq ChY_{k}I(t^{''}_{k+1} > T) = ChY_{\chi-1}, \;\;a.s.
\end{align*}
where $C > 0$ is independent of $h$.
\end{lemma}
\begin{proof}
Here, note that $k+1 =  \chi$ and $t_\chi = T$ (see \eqref{sticky_eqn_3.13}). Using Taylor' formula in time-variable and \eqref{wbp1}, we get
\begin{align}
    u_{k+1} - u^{'}_{k+1} & =  u(T, X_{\chi}) - u(T, X_{k+1}^{'}) - (t^{'}_{k+1} -T)\frac{\partial }{\partial t}u(T, X_{k+1}^{'})  + \mathcal{O}(h) \nonumber \\  &
    =  u(T, X_{\chi}) - u(T, X_{k+1}^{'}) - 2 p_{k+1}\mu^{\pi}_{k+1}\mathcal{A}\varphi^{\pi}_{k+1} \nonumber   \\  &  \;\;\;\; -  2 p_{k+1}\mu^{\pi}_{k+1} c(T, X^{\pi}_{k+1}) \varphi^{\pi}_{k+1} -  2 p_{k+1}\mu^{\pi}_{k+1} g(T, X^{\pi}_{k+1})   + \mathcal{O}(h), 
\end{align}
where $ p_{k+1} = (T-t^{'}_{k+1})/(2\mu_{k+1}^{\pi})$ (see (\ref{sticky_eqn_3.16})). 
Again, using Taylor's expansion with \eqref{sticky_eqn_3.13}, we obtain 
\begin{align}
    u(T, X_{\chi}) - u(T, X_{k+1}^{'}) = 2r_{k+1}(\nabla u(T, X^{\pi}_{k+1}) \cdot \nu_{k+1}^{\pi}) + \mathcal{O}(r_{k+1}^3). 
\end{align}
Therefore,
\begin{align}
    (u(T, X_{\chi}) &- u(T, X_{k+1}^{'}))Y_{k+1}   =    2r_{k+1}(\nabla u(T, X^{\pi}_{k+1}) \cdot \nu_{k+1}^{\pi})  Y_{k+1} -  2 p_{k+1}\mu^{\pi}_{k+1}\mathcal{A}\varphi^{\pi}_{k+1} Y_{k+1}
    \nonumber \\  & \quad  - 2 p_{k+1}\mu^{\pi}_{k+1} c(T, X^{\pi}_{k+1}) \varphi^{\pi}_{k+1}Y_{k+1} -  2 p_{k+1}\mu^{\pi}_{k+1} g(T, X^{\pi}_{k+1})Y_{k+1} \nonumber  \\  & \quad + \mathcal{O}(h)Y_{k} \nonumber  \\  & 
    = 2r_{k+1}(\nabla u(T, X^{\pi}_{k+1}) \cdot \nu_{k+1}^{\pi})  Y_{k} -  2 p_{k+1}\mu^{\pi}_{k+1}\mathcal{A}\varphi^{\pi}_{k+1} Y_{k}
    \nonumber \\  & \quad  - 2 p_{k+1}\mu^{\pi}_{k+1} c(T, X^{\pi}_{k+1}) \varphi^{\pi}_{k+1}Y_{k} -  2 p_{k+1}\mu^{\pi}_{k+1} g(T, X^{\pi}_{k+1})Y_{k} + \mathcal{O}(h)Y_{\chi-1}. \label{sticky_eqn_4.16}
\end{align}
Using \eqref{sticky_eqn_3.14}, we get
\begin{align}
    u^{'}_{k+1}(Y_{k+1} - Y_{k+1}^{'}) & = \Big(u(T, X_{k+1}^{\pi}) + \mu^{\pi}_{k+1}p_{k+1}\frac{\partial u}{\partial t}(T, X^{\pi}_{k+1}) + r_{k+1}(\nabla u(T, X^{\pi}_{k+1}) \cdot \nu_{k+1}^{\pi}) + \mathcal{O}(h)\Big) \nonumber
    \\  & \quad \times  \big(2r_{k+1}\gamma(T, X^{\pi}_{k+1}) Y_k   + 2p_{k+1}\mu_{k+1}^{\pi}c(T, X_{k+1}^{\pi}) \big)    \nonumber 
    \\  & =  2r_{k+1}\gamma(T, X^{\pi}_{k+1}) u(T, X_{k+1}^{\pi})Y_k + \mathcal{O}(h)Y_{\chi-1}. \label{sticky_eqn_4.17}
\end{align}
Adding \eqref{sticky_eqn_4.16} and \eqref{sticky_eqn_4.17}, we get
\begin{align}
    u_{k+1}Y_{k+1} - u_{k+1}^{'} Y_{k+1}^{'} &= 2r_{k+1}(\nabla u(T, X^{\pi}_{k+1}) \cdot \nu_{k+1}^{\pi})  Y_{k} + 2r_{k+1}\gamma(T, X^{\pi}_{k+1}) u(T, X_{k+1}^{\pi})Y_k  \nonumber \\   & \;\;\;\; -  2 p_{k+1}\mu^{\pi}_{k+1}\mathcal{A}\varphi^{\pi}_{k+1} Y_{k}
     \nonumber \\  & \quad 
     -  2 p_{k+1}\mu^{\pi}_{k+1} g(T, X^{\pi}_{k+1})Y_{k} + \mathcal{O}(h)Y_{\chi-1}. \label{sticky_eqn_4.17}
\end{align}
From (\ref{sticky_eqn_3.16}), we have
\begin{align}
    Z_{k+1} - Z_{k+1}^{'} &= -2\mu_{k+1}^{\pi}(r_{k+1} - p_{k+1})\mathcal{A}\varphi^{\pi}_{k+1}Y_{k} +2 p_{k+1}\mu^{\pi}_{k+1} c(T, X^{\pi}_{k+1}) \varphi^{\pi}_{k+1}Y_{k} \nonumber  \\  & \quad + 2 p_{k+1}\mu^{\pi}_{k+1} g(T, X^{\pi}_{k+1})Y_{k} - 2r_{k+1}\psi(T, X^{\pi}_{k+1}),
\end{align}
which combining with \eqref{sticky_eqn_4.17} yields the following:
\begin{align*}
    |u_{k+1}Y_{k+1} + Z_{k+1} - u_{k+1}^{'}Y_{k+1}^{'} - Z_{k+1}^{'}   |I_{\bar{G}^{c}}(X_{k+1}^{'})I(t^{''}_{k+1} > T) \leq CY_{\chi-1}h.
\end{align*}

\end{proof}

The following lemma corresponds to the Case IV mentioned in Subsection~\ref{subsection3.2}.
\begin{lemma}\label{sticky_lemma_4.5}
Let Assumptions~\ref{boundary_assum}-\ref{ellip_assum} hold, then we have the following one-step approximation:
\begin{align*}
|u_{k+1}Y_{k+1} + Z_{k+1} - u_{k+1}^{'}Y_{k+1}^{'} - Z_{k+1}^{'}   |I_{\bar{G}^{c}}(X_{k+1}^{'})I(t^{'}_{k+1} > T) \leq Chr_{k+1}Y_{k},\;\;\; a.s. 
\end{align*}
where $C > 0$ is independent of $h$. %and $\chi$ denotes the number of steps for Algorithm~\ref{sym_euler_method_stk_diff}.

\end{lemma}
\begin{proof}
Using Taylor's formula, we have
\begin{align}
    u(T, X_{\chi}) - u_{k+1}^{'} = u(T, X_{\chi}) - u(T, X^{'}_{k+1}) - (t^{'}_{k+1} - T) \frac{\partial}{\partial t} u(T + \epsilon (t_{k+1}^{'} - T)),
\end{align}
and
\begin{align*}
    u(T, X_{\chi}) &= u(T, X^{\pi}_{k+1}) + r_{k+1} (\nabla u(T, X^{\pi}_{k+1}) \cdot \nu^{\pi}_{k+1}) + r^{2}_{k+1} D^2 u(T, X^{\pi}_{k+1})[\nu^{\pi}_{k+1}, \nu^{\pi}_{k+1}] + \mathcal{O}(r_{k+1}^{3}), \\
    u(T, X_{k+1}^{'}) &= u(T, X^{\pi}_{k+1}) - r_{k+1} (\nabla u(T, X^{\pi}_{k+1}) \cdot \nu^{\pi}_{k+1}) + r^{2}_{k+1} D^2 u(T, X^{\pi}_{k+1})[\nu^{\pi}_{k+1}, \nu^{\pi}_{k+1}]+ \mathcal{O}(r_{k+1}^{3}).
\end{align*}
implying
\begin{align}
    u(T, X_{\chi}) - u(T, X^{'}_{k+1}) = 2r_{k+1}(\nabla u(T, X^{\pi}_{k+1}) \cdot  \nu^{\pi}_{k+1}) + \mathcal{O}(r_{k+1}^{3}).
\end{align}
Therefore, we have
\begin{align}
   (u_{k+1} - u_{k+1}^{'})Y_{k+1} &=  2r_{k+1}(\nabla u(T, X^{\pi}_{k+1}) \cdot  \nu^{\pi}_{k+1})Y_{k+1} + \mathcal{O}(hr_{k+1}) \nonumber \\  &
   = 2r_{k+1}(\nabla u(T, X^{\pi}_{k+1}) \cdot  \nu^{\pi}_{k+1})\big(Y_{k} + h c_k Y_k + 2 r_{k+1}\gamma(T, X^{\pi}_{k+1})Y_k \nonumber\\  &  \;\;\;\; + 2 r_{k+1}^2\gamma^2(T, X^{\pi}_{k+1})Y_k\big) + \mathcal{O}(hr_{k+1})
  \nonumber  \\  &
   = 2r_{k+1}(\nabla u(T, X^{\pi}_{k+1}) \cdot  \nu^{\pi}_{k+1})Y_{k}
   + 4r_{k+1}^2\gamma(T, X^{\pi}_{k+1})(\nabla u(T, X^{\pi}_{k+1}) \cdot \nu^{\pi}_{k+1}) \nonumber
   \\   & \;\;\;\;
   + \mathcal{O}(hr_{k+1}).  \label{sticky_eqn_4.23}
\end{align}
Similarly, using \eqref{sticky_euler_eqn_3.17}, we get
\begin{align}
    u^{'}_{k+1}(Y_{k+1} - Y_{k+1}^{'}) &= \big(u(T, X^{\pi}_{k+1}) - r_{k+1}(\nabla u(T, X_{k+1}^{\pi}) \cdot \nu^{\pi}_{k+1})\big)
    \nonumber \\ & \quad \times \big( 2 r_{k+1}\gamma(T, X^{\pi}_{k+1})Y_k+ 2 r_{k+1}^2\gamma^2(T, X^{\pi}_{k+1})Y_k\big)
    \nonumber 
    \\   & 
    = 2r_{k+1}u(T, X^{\pi}_{k+1})\gamma(T, X^{\pi}_{k+1})Y_{k} - 2r^2_{k+1}\gamma(T, X^{\pi}_{k+1})(\nabla u(T, X_{k+1}^{\pi}) \cdot \nu^{\pi}_{k+1})Y_{k}
    \nonumber \\ &  \;\;\;\; +  2r_{k+1}^2u(T, X_{k+1}^{\pi}) \gamma^2(T, X^{\pi}_{k+1})Y_k.\label{sticky_eqn_4.24}
\end{align}
Adding \eqref{sticky_eqn_4.23} and \eqref{sticky_eqn_4.24} gives
\begin{align}
    u_{k+1}Y_{k+1} &- u^{'}_{k+1}Y_{k+1}^{'} = 2r_{k+1}[(\nabla u(T, X^{\pi}_{k+1}) \cdot  \nu^{\pi}_{k+1}) + u(T, X^{\pi}_{k+1})\gamma(T, X^{\pi}_{k+1})]Y_{k} 
  \nonumber   \\   & \;\;\; \; 
   + 2r_{k+1}^2\gamma(T, X^{\pi}_{k+1})[(\nabla u(T, X^{\pi}_{k+1}) \cdot \nu^{\pi}_{k+1}) 
    +  u(T, X_{k+1}^{\pi}) \gamma(T, X^{\pi}_{k+1})]Y_k + \mathcal{O}(hr_{k+1}).\label{sticky_eqn_4.7}
\end{align}
Using \eqref{sticky_euler_eqn_3.18}, we have
\begin{align}
    Z_{k+1} - Z_{k+1}^{'} &=  -2r_{k+1}\mu^{\pi}_{k+1}\mathcal{A}\varphi^{\pi}_{k+1} -2r^{2}_{k+1}\gamma(T, X^{\pi}_{k+1})\mu^{\pi}_{k+1}\mathcal{A}\varphi^{\pi}_{k+1} - 2r_{k+1}\psi(T, X^{\pi}_{k+1}) \nonumber \\  &  \;\;\;\; - 2r_{k+1}\gamma(T, X^{\pi}_{k+1})\psi(T, X^{\pi}_{k+1})\label{sticky_eqn_4.8}
\end{align}
From \eqref{sticky_eqn_4.7}-\eqref{sticky_eqn_4.8}, we get
\begin{align*}
    |u_{k+1}Y_{k+1} + Z_{k+1} - u_{k+1}^{'}Y_{k+1}^{'} - Z_{k+1}^{'}   |I_{\bar{G}^{c}}(X_{k+1}^{'})I(t^{'}_{k+1} > T) \leq Chr_{k+1}Y_{k}, 
\end{align*}
where $C>0$ is independent of $T$. 
\end{proof}

\subsection{Proof of Theorem~\ref{thrm_sticky_euler_conv}}\label{sticky_subsec_conv_proof}

In addition to the one-step lemmas proved in the previous section, other important ingredient of the main convergence theorem is the lemma related to average number of hits of the approximating Markov chain with the boundary, which we prove below. To this end, we introduce a one-step transition operator for a general  Markov chain $(t_{k}^{'},X_{k}^{'})$, $k=0,\ldots,\hat{\chi}$, denoted with $\mathbb{T}_h=\mathbb{T}$ and defined as
\begin{equation*}
    (\mathbb{T}_h V)(t,x) = \mathbb{T}V(t,x) := \mathbb{E}\big[ \big . V(t_1^{'}, X_{1}^{'}) \big | t_{0}=t, X_{0}=x\big],
\end{equation*}
where $(t,x)$ is an arbitrary point in $[T_{0}, T_1)\times G$
and $V(t,x) :  [T_{0}, T_1]\times \big(\bar G \cup \bar{G}_{-r}\big) \rightarrow \mathbb{R}$.
We note that
\begin{equation*}
\mathbb{E}\big[ \big . V(t_{k+1}^{'}, X_{k+1}^{'}) \big | X_{k}^{'}=x\big]
=\mathbb{T} V(t_k^{'},x).
\end{equation*}

Consider the boundary value problem associated with a Markov chain $(t_{k}^{'},X_{k}^{'})$:
\begin{align}
    q(x) \mathbb{T}V(t,x) - V(t,x) &= -f(t,x),  &(t,x) \in [T_{0}, T_1)\times \big(\bar{G}\cup \bar{G}_{-r}\big), \label{eq:18} \\
    V(T,x) &= 0,       &x \in  \big(\bar{G} \cup \bar{G}_{-r}\big),\label{eq:19}
\end{align}
where $g(t,x) \geq 0$ and $q(x) > 0$. The solution to this problem starting from $(t,x)=(t_k,x)$ is given by \cite{20, 46, mil_tretyakov_book}:
\begin{equation}
    V(t_k^{'},x) =  \mathbb{E}\bigg[ \bigg . \sum\limits_{i=k}^{\hat\chi-1}f(t_{i}^{'},X_{i}^{'})\prod\limits_{j=k}^{i-1}q(X_{j}^{'}) \bigg | X_{k}^{'}=x\bigg]. \label{eq:20}
\end{equation}

Consider the Markov chain $(t_{k}^{'}, X^{'}_{k})_{k=0}^{\chi -2}$ constituting auxiliary steps in Algorithm~\ref{sticky_euler_method}. Consider a $d-1$ dimensional hypersurface $\partial G_{-r}$ which is parallel to $\partial G$  and at  distance of $r$, where $r = \mathcal{O}(h^{1/2})$, such that $r_k < r$ for all $k$ steps in Algorithm~\ref{sticky_euler_method}. We denote the boundary zone between $\partial G$ and $\partial G_{-r}$ as $G_{-r}$.  For the convenience, we mention the update rule again here, for $k = 0, \dots, \chi-2 $,  $t_{k+1}^{'} = t_{k}^{'} + 2r_{k}\mu(X_{k}^{\pi})I_{\bar{G}_{-r}}(X_{k}^{'}) + h$ and $X_{k+1}^{'}  =  X_{k}^{'} + 2r_{k}\nu^{\pi}_{k} I_{\bar{G}_{-r}}(X_{k}^{'}) + \delta(X^{'}_{k})  $ with $\delta (x) =  b(x + 2 r I_{\bar{G}_{-r}}(x)) h + \sigma (x + 2 r I_{\bar{G}_{-r}}(x))\xi \sqrt{h}$.   
\begin{lemma}\label{sticky_lemma_avgreflection}
    Let Assumptions~\ref{boundary_assum}-\ref{sticky_assum} hold. Then, the following inequality is satisfied for a fixed $K>0$:
    \begin{align*}
       \mathbb{E}\sum_{k=0}^{\chi-2}r_{k}I_{\bar{G}^c}(X_{k}^{'}) \prod_{i=1}^{k-1}(1 + K r_{i}) \leq C, 
    \end{align*}
    where $C>0$ is independent of $h$.
\end{lemma}
\begin{proof}
    If we take $q (x) := 1 + K\dist(x, \partial G) I_{\bar{G}^{c}}(x)$ and $ f(t,x) = \dist(x, \partial G) $ then solution to \eqref{eq:18}-\eqref{eq:19} is given by
    \begin{align*}
        v(t_0, x) = \mathbb{E}\sum_{k=0}^{\chi-2}r_{k}I_{\bar{G}^c}(X_{k}^{'}) \prod_{i=1}^{k-1}(1 + K r_{i}).
    \end{align*}
If we fix $q(x) $ as above and find a solution  $V(t,x) $  of \eqref{eq:18}-\eqref{eq:19} for a $f$ which satisfies $f(t,x) \geq I_{\bar{G}^c}(x)r(x)$ for all $(t,x) \in [t_0, T]\times \bar{G} \cup \bar{G}_{-r}$ then $v(t_0,x) \leq V(t_0, x) $. 
    
 To this end, consider the following function: 
 \begin{align}
     W(t, x) &= \begin{cases}
         e^{K_1 (T-t)} e^{K_2 B(x)},\quad  (t,x) \in [t_0, T_1] \times (\bar{G}\cup \bar{G}_{-r}),\\
         0, \quad (t,x) \in \{T_1\} \times (\bar{G}\cup \bar{G}_{-r}),
     \end{cases}
 \end{align}
where $B(x) = \dist(x,\partial G)$ for all $x \in G_{-r}$ and then we smoothly extend it to the entire space ensuring $B(x) \in C^4(\bar{G} \cup \bar{G}_{-r})$. Note that such an extension is possible due to Assumption~\ref{boundary_assum} (see \cite[Lemma~6.37]{gilbarg_trudinger}). We discuss the choice of $K_1$ and $K_2$ later in the proof. 

 We consider two cases here. In first case, we calculate $f(t,x) := - ( \mathbb{T} W(t,x) - W(t,x))$ when $x \in \bar{G}$. Nest, we discuss  the scenario when $x \in \bar{G}^c$. 
 \begin{description}
     \item[Case 1: $x \in \bar{G}$. ] In this case, $q(x) = 1$ and we have
     \begin{align*}
         \mathbb{T}&W(t,x) - W(t,x) = \mathbb{E}e^{K_1(T-t-h)}e^{K_2 B(X_1^{'})} - e^{K_1(T-t)}e^{K_2 B(x)} 
         \\  & = e^{K_1(T-t)}e^{K_2 B(x)}(1 - K_1 h + \alpha_1(K_1)\mathcal{O}(h^2)) \big( 1 + K_2 \mathbb{E} (\nabla B(x) \cdot\delta(x)) + \alpha_1(K_2)\mathcal{O}(h)\big)  \\  & \quad - e^{K_1(T-t)}e^{K_2 B(x)} 
         \\  & 
         = (- K_1 h + K_2 \mathcal{O}(h) + \alpha_3(K_1, K_2)\mathcal{O}(h^2))e^{K_1(T-t)}e^{K_2 B(x)},
     \end{align*}
     where $\mathcal{O}(h)$ and $\mathcal{O}(h^2)$ do not depend on $K_1$, and $\alpha_i$, $i=1,\dots,3$ are generic functions.
\item[Case 2: $x \in \bar{G}^c$. ]
In this case, $q(x) = 1 + K\dist(x ,\partial G)$. For brevity, we denote $ r(x) := \dist(x ,\partial G)$ and recall $x^{\pi}$ is projection of $x$ on $\partial G$. We have
\begin{align*}
    (&1 + K r(x))  \mathbb{T}W(t,x)  - W(t,x) = (1 + Kr(x)) e^{K_1 (T- t - 2r(x) \mu(x^{\pi}) - h)} e^{K_2 B(X_{1}^{'})}  -  e^{K_1(T-t)}e^{K_2 B(x)} 
    \\  & = e^{K_1(T-t)}e^{K_2 B(x)}\bigg[(1 + Kr(x)) \big(1 - K_1 r(x) \mu(x^{\pi}) - K_1 h + \frac{K_{1}^{2}}{2} r^2(x)\mu^2(x^{\pi}) + K_1^2\mathcal{O}( h^{3/2})\big)
    \\  &  \quad \;\;  \times \bigg( 1 + K_2 r(x)\bigg( \frac{x - x^{\pi}}{|x - x^{\pi}|}\cdot \nu(x^{\pi})\bigg) + K_2^2 \mathcal{O}(r^2(x))\bigg) -1 \bigg] 
    \\  & 
    = -e^{K_1(T-t)}e^{K_2 B(x)} \underbrace{\big(  K_1 r(x) \mu(x^{\pi}) + K_2 r(x) - K r(x) \big)}_{(*)} \\  & \quad  - e^{K_1(T-t)}e^{K_2 B(x)} \underbrace{\big( K_1 h - \alpha_4(K_2)\mathcal{O}(h) -  \alpha_5(K_1, K_2) \mathcal{O}(h^{3/2}) + K_1 K r(x) h + K_2 K r^2(x) \big)}_{(**)},  
\end{align*}
where $\alpha_i$, $i=4,5$ are generic functions. 
 \end{description} We choose $K_2 = K+2$ such that term (*) in Case~2  is always greater than $r(x)$. Depending on $K_2$, we make a choice of $K_1$ such that term (**) and $f(t,x) = -(\mathbb{T}W- W)$  in Case~1 are positive . This implies that we have constructed a function $f(t,x) = - (1 + I_{G_{-r}}(x))(\mathbb{T} W - W)$ such that $f(t,x) \geq I_{\bar{G}^{c}}(x) r(x)$. Therefore, $v(t_0, x) \leq W(t_0,x)$.  
\end{proof}

\begin{corollary} \label{sticky_lemma_avgreflection_corollary1}
Under Assumptions~\ref{boundary_assum}-\ref{sticky_assum}, for any constant $K>0$ the following inequalities hold
\begin{equation}\label{neweq3.29}
\mathbb{E}\bigg(\sum\limits_{k=1}^{\chi-2}\prod_{i=0}^{k-1}(1+Kr_{i}I_{G_{-r}}(X_{i}^{'}))\bigg) \leq \frac{C}{h},\quad\text{and},\quad
    \mathbb{E}\bigg( \prod_{i=0}^{\chi-2}(1+Kr_{i}I_{G_{-r}}(X_{i}^{'}))  \bigg) \leq C,
\end{equation}
where $C$ is a positive constant independent of $h$.
\end{corollary}
\begin{proof}
We return to the boundary value problem (\ref{eq:18})-(\ref{eq:19}) associated with the Markov chain $(t_{k}^{'},X_{k}^{'})$, whose solution is given by (\ref{eq:20}). If we take $f(t,x) = 1$ and $q(x) = (1+r_{0}{I_{\bar{G}_{-r}}(x)})$ then the solution
 of the problem is
\begin{equation}\label{3.27}
     v(t_0,x) = \mathbb{E}\Bigg( \bigg . \sum\limits_{k=0}^{\chi-2}\prod_{i=0}^{k-1}\big(1+Kr_{i}I_{G_{-r}}(X_{i}^{'})\big) \bigg |  X_{0}^{'}=x \Bigg).
\end{equation}
If we can construct a solution to (\ref{eq:18})-(\ref{eq:19}) with $f(t,x) \geq 1$, then $v(t,x) \leq V(t,x)$. To achieve this, we choose $f(t,x) = l(t,x)/h$, where $l(t,x)$ is constructed according to the cases in Lemma \ref{sticky_lemma_avgreflection}. This choice ensures $f(t,x) \geq 1$, yielding $v(t_0,x) \leq V(t_0,x)$ with $V(t,x) = W(t,x)/h$. Our next goal is to prove that $\mathbb{E}\big(\prod_{i=0}^{\chi-2}(1+Kr_{i}I_{G_{-r}}(X_{i}^{'}))\big) \leq C$, where the constant $C$ does not depend on $h$. Observe that
\begin{align*}
    &\mathbb{E}\bigg(\sum\limits_{k=0}^{\chi-2}r_{k}I_{G_{-r}}(X_{k}^{'})\prod_{i=0}^{k-1}(1+ Kr_{i}I_{G_{-r}}(X_{i}^{'}))\bigg)
     = \frac{1}{K}\mathbb{E}\bigg(\sum\limits_{k=0}^{\chi-2}\Big( \prod\limits_{i=0}^{k}(1+Kr_{i}I_{G_{-r}}(X_{i}^{'})) \\ & - \prod\limits_{i=0}^{k-1}(1+Kr_{i}I_{G_{-r}}(X_{i}^{'}))\Big)\bigg)
     =\frac{1}{K}\mathbb{E}\bigg(\prod\limits_{i=0}^{\chi-2}\big(1+Kr_{i}I_{G_{-r}}(X_{i}^{'})\big) - 1\bigg).
\end{align*}
Using Lemma~\ref{sticky_lemma_avgreflection}, we obtain second bound in (\ref{neweq3.29}). 
\end{proof}

\begin{proof}[Proof of Theorem~\ref{thrm_sticky_euler_conv}]
We have
\begin{align*}
    \mathbb{E}&\varphi (X_\chi)Y_{\chi} + Z_{\chi} - \mathbb{E}[\varphi(X(T))Y(T)+ Z(T)] =   \mathbb{E}u(t_\chi, X_{\chi})Y_\chi + Z_{\chi} - u(t_0, x)Y_{0} - Z_{0} 
    \\ & = \mathbb{E}\sum_{k=0}^{\chi -1} u_{k+1}Y_{k+1} + Z_{k+1} - u_{k}Y_{k} - Z_{k}  =  
    \mathbb{E}\sum_{k=0}^{N - 1}
    \mathbb{E} (u_{k+1}Y_{k+1} + Z_{k+1} -u_{k}Y_{k} - Z_{k} \; |\; X_{k}),
\end{align*}
where  we assign $u_{k +1} := u_k$, $Y_{k+1} := Y_k$ and $Z_{k+1} := Z_k$ for $ k = \chi,\dots, N := \lfloor T/h \rfloor$. It is not difficult to deduce from Algorithm~\ref{sticky_euler_method} that $\chi \leq N$. Splitting the terms, we obtain 
\begin{align*}
    \mathbb{E}\varphi (X_\chi)Y_{\chi} + Z_{\chi} - \mathbb{E}[\varphi(X(T))Y(T) + Z(T)] &=  \mathbb{E}\sum_{k=0}^{\chi -1}
 (u_{k+1}Y_{k+1} + Z_{k+1}-u^{'}_{k+1}Y^{'}_{k+1} - Z_{k+1}^{'}) \\  & \quad  +  \mathbb{E}\sum_{k=0}^{N -1}
    \mathbb{E} (u^{'}_{k+1}Y^{'}_{k+1} + Z^{'}_{k+1} -u_{k}Y_{k} - Z_{k} \; |\; X_{k}).
    \end{align*}
It can be shown via Taylor's expansion that the following holds under Assumptions~\ref{boundary_assum}-\ref{ellip_assum}: 
\begin{align}
    \mathbb{E}(u_{k+1}^{'} Y_{k+1}^{'} + Z_{k+1}^{'} - u_{k} Y_{k} - Z_{k} \;|\; X_k)I(t_{k+1}^{'} < T) \leq CY_{k}h^2, \label{one_step_sticky_inside _domain}
\end{align}
where $C >0$ is  independent of $h$. The proof of above one-step error estimate is available in \cite{mil_tretyakov_book} (also see \cite[Lemma~3.2]{leimkuhler2023simplerandom}). Therefore, we obtain
    \begin{align*}
  |\mathbb{E}\varphi (X_\chi)Y_{\chi} + Z_{\chi} - \mathbb{E}[\varphi(X(T))Y(T) + Z(T)]|   &\leq  \bigg|\mathbb{E}\sum_{k=0}^{\chi -1}
     (u_{k+1}Y_{k+1} + Z_{k+1}-u^{'}_{k+1}Y^{'}_{k+1} -Z_{k+1}^{'})\bigg|  \\   &  \;\;\; + Ch^2 \mathbb{E}\sum_{k=0}^{\chi-1}Y_{k},
\end{align*}
where $C >0$ is independent of $h$. We split the term as follows:
\begin{align*}
  \sum_{k=0}^{\chi -1}  |u_{k+1}Y_{k+1} &+ Z_{k+1} - u_{k+1}^{'}Y_{k+1}^{'} - Z_{k+1}^{'}   |I_{\bar{G}^{c}}(X_{k+1}^{'}) 
    \\   & = |u_{k+1}Y_{k+1} + Z_{k+1} - u_{k+1}^{'}Y_{k+1}^{'} - Z_{k+1}^{'}   |I_{\bar{G}^{c}}(X_{k+1}^{'})I(t^{'}_{k+1} \geq T)   \\  
    & \quad + \sum_{k=0}^{\chi -1} |u_{k+1}Y_{k+1} + Z_{k+1} - u_{k+1}^{'}Y_{k+1}^{'} - Z_{k+1}^{'}   |I_{\bar{G}^{c}}(X_{k+1}^{'})I(t^{'}_{k+1} < T) 
    \\   & = |u_{k+1}Y_{k+1} + Z_{k+1} - u_{k+1}^{'}Y_{k+1}^{'} - Z_{k+1}^{'}   |I_{\bar{G}^{c}}(X_{k+1}^{'})I(t^{'}_{k+1} \geq T)   \\  
    & \quad + \sum_{k=0}^{\chi -2} |u_{k+1}Y_{k+1} + Z_{k+1} - u_{k+1}^{'}Y_{k+1}^{'} - Z_{k+1}^{'}   |I_{\bar{G}^{c}}(X_{k+1}^{'})I(t^{'}_{k+1} < T) I(t_{k+1}^{''}< T) 
    \\  
    & \quad + |u_{k+1}Y_{k+1} + Z_{k+1} - u_{k+1}^{'}Y_{k+1}^{'} - Z_{k+1}^{'}   |I_{\bar{G}^{c}}(X_{k+1}^{'})I(t^{'}_{k+1} < T) I(t_{k+1}^{''} \geq T). 
\end{align*}
Using Lemma~\ref{sticky_lemma_4.2}, Lemma~\ref{sticky_lemma_4.3} and Lemma~\ref{sticky_lemma_4.4}, we get  
\begin{align*}
     \sum_{k=0}^{\chi -1}  |u_{k+1}Y_{k+1} + Z_{k+1} &- u_{k+1}^{'}Y_{k+1}^{'} - Z_{k+1}^{'}   |I_{\bar{G}^{c}}(X_{k+1}^{'})  \leq ChY_{\chi-1} + Ch\sum_{k=0}^{\chi -2}r_{k+1}I_{\bar{G}^c}(X_{k+1}^{'})Y_{k}.
\end{align*}
Therefore, 
\begin{align*}
      |\mathbb{E}\varphi (X_\chi)Y_{\chi} + Z_{\chi} - \mathbb{E}[\varphi(X(T)) Y(T) + Z(T)]|   \leq  Ch \mathbb{E} Y_{\chi-1} + Ch\mathbb{E}\sum_{k=0}^{\chi -2} r_{k+1}Y_{k}I_{\bar{G}^c}(X_{k+1}^{'})
       + Ch^2 \mathbb{E}\sum_{k=0}^{\chi-1}Y_{k} .
\end{align*}
From \eqref{sticky_eulerscheme_y}, \eqref{sym_y}, \eqref{sticky_eqn_3.14} and \eqref{sticky_euler_eqn_3.17}, it is not difficult to infer that there exist $K, \tilde K >0$ independent of $h$ such that
\begin{align*}
    Y_{k} &\leq Y_{k-1} ( 1 + Kr_{k} + \tilde{K} h)  \leq Y_{k-1} ( 1 + Kr_{k}I_{\bar{G}^c}(X_{k}^{'}) )(1 + \tilde{K} h) \leq
    Y_{0} e^{\tilde{K} T}\prod_{i=0}^{k-1}(1+ K r_{i}
I_{\bar{G}^c}(X_{i}^{'})).\end{align*}
This implies
\begin{align*}
      |\mathbb{E}\varphi (X_\chi)Y_{\chi} + Z_{\chi} - \mathbb{E}[\varphi(X(T))Y(T) + Z(T)]|  &  \leq  Ch \mathbb{E} \prod_{k=0}^{\chi-2}Y_{k} + Ch^2 \mathbb{E}\sum_{k=0}^{\chi-1}\prod_{i=0}^{k-1}(1+ K r_{i}I_{\bar{G}^c}(X_{i}^{'})) 
     \\  &  \quad  +  Ch\mathbb{E}\sum_{k=0}^{\chi -2} r_{k+1}I_{\bar{G}^c}(X_{k+1}^{'})\prod_{i=0}^{k-1}(1+ K r_{i} I_{\bar{G}^c}(X_{i}^{'}))
\end{align*}
which on applying Lemma~\ref{sticky_lemma_avgreflection} and Corollary~\ref{sticky_lemma_avgreflection_corollary1} give the first-order of convergence.  
 \end{proof}

\section{Convergence of Projected Euler scheme for sticky SDEs}\label{sticky_proj_euler_section}

We first state the main theorem of convergence for projected Euler scheme i.e. Algorithm~\ref{proj_euler_method_stk_diff_algo}. 

\begin{theorem}\label{proj_conv_thrm}
Under Assumptions~\ref{boundary_assum}-\ref{ellip_assum}, the weak order of convergence of  \textbf{Algorithm~\ref{proj_euler_method_stk_diff_algo}} is $\mathcal{O}(h^{1/2})$ i.e. 
\begin{align}
|\mathbb{E} (\varphi(X_{\chi})Y_{\chi} + Z_{\chi}) - u(t_{0}, X_{0})| \leq Ch^{1/2}, 
\end{align}
where $ u(t,x) $ is solution of (\ref{wbp1})-(\ref{wbp3}), and $C $ is a positive constant independent of $h$. 
\end{theorem}

We now prove one-step lemma for error near the boundary. 

\begin{lemma}[One-step lemma for Algorithm~\ref{proj_euler_method_stk_diff_algo}]\label{proj_euler_one_step}
Let Assumptions~\ref{boundary_assum}-\ref{ellip_assum} hold. Then the one-step error of Algorithm~\ref{proj_euler_method_stk_diff_algo} when $X_{k} \notin \bar{G}$ is given by
\begin{align*}
|u_{k+1}Y_{k+1} + Z_{k+1} - u_{k}Y_{k} - Z_{k}| \leq Cr_{k+1}^{2}Y_{k},\;\;\; k=0,\dots,\chi-2,\;\;\; a.s.
\end{align*}
where $C$ is a positive constant independent of $h$ and $Y_k$. 
\end{lemma}
\begin{proof}
We split the terms as follows: 
\begin{align}
u_{k+1}Y_{k+1} + Z_{k+1} - u_{k}Y_{k} - Z_{k} = (u_{k+1} - u_{k})Y_{k+1} + u_{k}(Y_{k+1} - Y_{k}) + Z_{k+1} - Z_{k}.   
\end{align}
Using Taylor's expansion, we obtain
\begin{align*}
u_{k+1} = u(t_{k}, X_{k+1}) + h_{k}\frac{\partial u}{\partial t} (t_{k}, X_{k+1})  + \frac{h_{k}^{2}}{2}\frac{\partial^2}{\partial t^2} u(t_{k} + \epsilon_{1} h_{k}, X_{k+1}),
\end{align*}
where $\epsilon_1 \in (0, 1)$ and  recall $h_k = \mu(X_{k+1})r_{k}$. Again, using Taylor's expansion, we ascertain
\begin{align*}
u_{k} & = u(t_k, X_{k+1} -r_{k}\nu_{k+1}) = u(t_{k}, X_{k+1}) - r_{k} (\nabla u(t_{k}, X_{k+1}) \cdot \nu_{k+1}) \\ &  \quad + \frac{r_{k}^{2}}{2}D^{2}u(t_{k}, X_{k+1} - \epsilon_{2}r_{k}\nu_{k+1})[\nu_{k+1}, \nu_{k+1}], 
\end{align*}
where $\epsilon_{2} \in (0,1)$ and $\nu_{k+1} := \nu(X_{k+1})$. Hence,
\begin{align}
(u_{k+1}&-u_{k})Y_{k+1} = \Big(h_{k}\frac{\partial u}{\partial t}(t_{k}, X_{k+1})  + \frac{h_{k}^{2}}{2}\frac{\partial^2}{\partial t^2} u(t_{k} + \epsilon_{1} h_{k}, X_{k+1}) + r_{k} (\nabla u(t_{k}, X_{k+1}) \cdot \nu_{k+1}) \nonumber \\ &- \frac{r_{k}^{2}}{2}\nabla^{2}u(t_{k}, X_{k+1} - \epsilon_{2}r_{k}\nu_{k+1})[\nu_{k+1}, \nu_{k+1}]\Big)(Y_{k} + r_{k}c(t_{k}, X_{k+1})Y_{k} + r_{k}\gamma(t_{k}, X_{k+1})Y_{k}\Big). \label{stk_eq_4.3}
\end{align}
Since $X_{k+1} \in \bar{G}$, using (\ref{wbp1}), we  have
\begin{align}
    \frac{\partial}{\partial t}u(t_{k}, X_{k+1}) & = - (b(t_{k}, X_{k+1}) \cdot \nabla u(t_{k}, X_{k+1})) - \big(a(t_{k}, X_{k+1}) : \nabla^2 u(t_{k}, X_{k+1}) \big)\nonumber \\ & \;\;\;\; - c(t_{k},X_{k+1})u(t_{k}, X_{k+1}) - g(t_{k}, X_{k+1}). \label{stk_eq_4.4}
\end{align}
Combining (\ref{stk_eq_4.3}) and (\ref{stk_eq_4.4}) yields
\begin{align}
(u_{k+1}&-u_{k})Y_{k+1} = \bigg(-r_{k}\mu(X_{k+1})\Big((b(t_{k}, X_{k+1}) \cdot \nabla u(t_{k}, X_{k+1})) - \big(a(t_{k}, X_{k+1}) : \nabla^2 u(t_{k}, X_{k+1})\big) \Big) \nonumber \\ &  \;\;\; \;+  r_{k}Du(t_{k}, X_{k+1})[\nu_{k+1}]\bigg)Y_{k} - r_{k}\mu(X_{k+1})c(t_{k}, X_{k+1})u(t_{k}, X_{k+1})Y_{k} \nonumber  \\&  \;\;\; \; - r_{k}\mu(X_{k+1})g(t_{k}, X_{k+1})Y_{k} + \mathcal{O}(r_{k}^{2}Y_{k})
\nonumber\\ &  = -r_{k}\gamma(t_{k}, X_{k+1})u(t_{k}, X_{k+1})Y_{k} + r_{k}\psi(t_{k}, X_{k+1})Y_{k}- r_{k}\mu(X_{k+1})c(t_{k}, X_{k+1})u(t_{k}, X_{k+1})Y_{k} \nonumber \\ &  \;\;\; \; - r_{k}\mu(X_{k+1})g(t_{k}, X_{k+1})Y_{k} + \mathcal{O}(r_{k}^{2}Y_{k}), \label{proj_eqn_4.5}
\end{align}
where we have used (\ref{wbp3}). From (\ref{proj_y}), we have
\begin{align}
u_{k}(Y_{k+1} - Y_{k}) &= (u(t_{k},X_{k+1}) - r_{k}Du(t_{k}, X_{k+1} - \epsilon_{3}r_{k}\nu_{k+1})[\nu_{k+1}]) \nonumber  \\ & \;\;\;\; \times \big(r_{k}\gamma(t_{k}, X_{k+1})Y_{k} + r_{k}c(t_{k}, X_{k+1})Y_{k}\big), \label{proj_eqn_4.6}
\end{align}
where $\epsilon_3 \in (0,1)$. Adding (\ref{proj_eqn_4.5}) and (\ref{proj_eqn_4.6}) gives
\begin{align}
u_{k+1} Y_{k+1} -  u_{k}Y_{k} = r_{k}\psi(t_{k},X_{k+1})Y_{k} - r_{k}\mu(X_{k+1})g(t_{k},X_{k+1})Y_{k} +\mathcal{O}(r_{k}^2 Y_k), 
\end{align}
 which on using the expression for $(Z_{k+1} - Z_{k})$ from (\ref{proj_z}) gives the desired bound.
\end{proof}

\begin{proof}[Proof of Theorem~\ref{proj_conv_thrm}]
We have
\begin{align*}
&\mathbb{E}(\varphi(X_{\chi})Y_{\chi}  + Z_{\chi} - u(t_{0}, x)) = \mathbb{E}\Big(\sum\limits_{k=0}^{\chi -1 }u_{k+1}Y_{k+1} + Z_{k+1} - u_{k}Y_{k} - Z_{k})\Big) \\ & = \mathbb{E}\Big(\sum\limits_{k=0}^{\chi -1 }\big(u_{k+1}Y_{k+1} + Z_{k+1} - u_{k}Y_{k} - Z_{k})\big)I_{\bar{G}}(X_{k})\Big)  
+ \mathbb{E}\Big(\sum\limits_{k=0}^{\chi -1 }\big(u_{k+1}Y_{k+1} + Z_{k+1} - u_{k}Y_{k} - Z_{k})\big)I_{\bar{G}^{c}}(X_{k})\Big).
\end{align*}
 Note that due to uniform boundedness of the terms involved inside summation, we have
\begin{align*}
\mathbb{E}&\Big(\sum\limits_{k=0}^{\chi -1 }\big(u_{k+1}Y_{k+1} + Z_{k+1} - u_{k}Y_{k} - Z_{k})\big)I_{\bar{G}}(X_{k})\Big)  = \sum\limits_{k=0}^{\infty }\mathbb{E}\big((u_{k+1}Y_{k+1} + Z_{k+1} - u_{k}Y_{k} - Z_{k})I_{\bar{G}}(X_{k})\big),
\end{align*}
where we assign $X_{k+1} = X_{k}$, $t_{k+1} = t_{k}$, $Y_{k+1} = Y_{k}$, and $Z_{k+1} = Z_{k}$ when $k > \chi$.
From \cite{leimkuhler2023simplerandom}, we know that $
 \mathbb{E}\big(u_{k+1}Y_{k+1} + Z_{k+1} - u_{k}Y_{k} - Z_{k} | X_{k} \in \bar{G}\big) \leq CY_{k}h^{2}
$. Therefore,
\begin{align*}
\mathbb{E}&\big((u_{k+1}Y_{k+1} + Z_{k+1} - u_{k}Y_{k}  - Z_{k})I_{\bar{G}}(X_{k})\big) = \mathbb{E}\Big( \mathbb{E}\big(u_{k+1}Y_{k+1} + Z_{k+1} - u_{k}Y_{k} - Z_{k} | X_{k} \in \bar{G}\big)I_{\bar{G}}(X_{k})\Big) \\ & = \mathbb{E}\Big( \mathbb{E}\big(u_{k+1}Y_{k+1} + Z_{k+1} - u_{k}Y_{k} - Z_{k} | X_{k} \in \bar{G}\big)I_{\bar{G}}(X_{k})I(\chi > k)\Big) \leq Ch^{2} \mathbb{E}(Y_{k}I_{\bar{G}}(X_{k})I(\chi > k) ),
\end{align*}
which implies 
\begin{align*}
 \sum\limits_{k=0}^{\infty }\mathbb{E}\big((u_{k+1}Y_{k+1} + Z_{k+1} - u_{k}Y_{k} - Z_{k})I_{\bar{G}}(X_{k})\big) \leq Ch^{2} \mathbb{E}\sum\limits_{k=0}^{\infty}(Y_{k}I_{\bar{G}}(X_{k})I(\chi > k) ) \leq  Ch^{2} \mathbb{E}\Big(\sum\limits_{k=0}^{\chi-1}Y_{k}\Big).
\end{align*}
From Lemma~\ref{proj_euler_one_step}, we have
\begin{align}
|u_{k+1}Y_{k+1} + Z_{k+1} - u_{k}Y_{k} - Z_{k}|I_{\bar{G}^{c}}(X_{k}) \leq Cr_{k}^{2}Y_{k} I_{\bar{G}^{c}}(X_{k}),
\end{align}
which implies
\begin{align}
 \mathbb{E}\Big(\sum\limits_{k=0}^{\chi -1 }\big(u_{k+1}Y_{k+1} + Z_{k+1} - u_{k}Y_{k} - Z_{k})\big)I_{\bar{G}^{c}}(X_{k})\Big) \leq  Ch^{1/2}\mathbb{E}\Big(\sum\limits_{k=0}^{\chi - 1}r_{k}Y_{k}I_{\bar{G}^{c}}(X_{k})\Big). 
\end{align}
Using \eqref{stickyproj_eulerscheme_y} and \eqref{proj_y}, we can deduce that
\begin{align*}
    Y_{k} \leq  Y_{k-1}(1 + K r_{k-1}I_{\bar{G}^c}(X_{k-1}))(1+ \tilde{K}h)  = Y_{0}e^{\tilde{K} T}\prod\limits_{i=0}^{k}(1 + K r_{i} I_{\bar{G}^c}(X_{i})),
\end{align*}
where $K, \tilde{K}>0$ are independent of $ h$.  Therefore,
\begin{align}
\mathbb{E}(\varphi(X_{\chi})Y_{\chi}  + Z_{\chi} - u(t_{0}, x)) &\leq Ch^{2} \mathbb{E}\sum_{k=0}^{\chi}\prod\limits_{i=0}^{k}(1 + K r_{i} I_{\bar{G}^c}(X_{i}))  
\nonumber \\   & \quad 
+ Ch^{1/2} \mathbb{E}\sum_{k=0}^{\chi}r_{k}I_{\bar{G}^{c}}(X_{k})\prod\limits_{i=0}^{k}(1 + K r_{i} I_{\bar{G}^c}(X_{i})). \label{sticky_new_eq_5.10}
\end{align}
Using the same set of arguments as in Lemma~\ref{sticky_lemma_avgreflection} and Corollary~\ref{sticky_lemma_avgreflection_corollary1}, it can be shown that
\begin{align*} 
\mathbb{E}\sum_{k=0}^{\chi}r_{k}I_{\bar{G}^{c}}(X_{k})\prod\limits_{i=0}^{k}(1 + K r_{i} I_{\bar{G}^c}(X_{i})) \leq C, \quad \text{and}\quad \mathbb{E}\sum_{k=0}^{\chi}\prod\limits_{i=0}^{k}(1 + K r_{i} I_{\bar{G}^c}(X_{i})) \leq C/h, 
\end{align*}
where $ C >0 $ is independent of $h$. Using the above bounds in \eqref{sticky_new_eq_5.10}, we get the desired result. 
\end{proof}

\section{Numerical Illustration}\label{sticky_sec_num_exp}

This section verifies the theoretical results for the numerical schemes presented in Section~\ref{sticky_num_method_section}   for finite-time weak-sense accuracy. 
 For the Monte Carlo approximation of $\mathbb{E}\varphi(X_{N})$ we use $\hat{\varphi}_{N} = \frac{1}{M}\sum_{i=1}^{M} \varphi(X^{(i)}_{N})$ where $X^{(i)}_{N}$ represent
 independent samples of the numerical solution.  Finally, ${D}_{M}$
denotes the usual sample variance associated with our Monte Carlo estimators.

For the considered numerical illustration, take $G = \{x_1 ^2 + x_2^2 \leq R^2\}$ with $\partial G = \{ x_1^2 + x_2^2 = R^2 \}$. Consider the following parabolic PDE:
\begin{align}
    \frac{\partial u}{\partial t} + \frac{x_1}{2} \frac{\partial u}{\partial x_1} + 2x_2\frac{\partial u}{\partial x_2} + \frac{1}{2}\frac{\partial^2 u}{\partial x_1^2 } + \frac{3 x_2^2}{2}\frac{\partial^2 u}{\partial x_2^2 }  + x_2 u - (1  + 2x_1^2 + 8x_2^2 + x_1^2 x_2 &+ x_2^3)e^{-(1-t)} - 10x_2 = 0, \nonumber \\  & (t,x) \in [0,1]\times \bar{G},  \nonumber
\end{align}
with sticky boundary condition:
\begin{align}
   2x_1^2 \bigg(\frac{x_1}{2} \frac{\partial u}{\partial x_1} + 2x_2\frac{\partial u}{\partial x_2} + \frac{1}{2}\frac{\partial^2 u}{\partial x_1^2 } + \frac{3 x_2^2}{2}\frac{\partial^2 u}{\partial x_2^2 } \bigg) - 0.5 u - (\nabla u \cdot x)/R & =  -5 - \frac{1}{10}x_1^2 - \frac{21}{10}x_2^2 + 2x_1^4 + 14x_1^2 x_2^2, \nonumber   \\ &   (t,x) \in [0,1]\times \partial G, \nonumber 
\end{align}
and terminal condition:
\begin{align*}
    u(1,x) = x_1^2 + x_2^2 + 10, \quad x \in \bar{G}. 
\end{align*}
The construction of the model problem in this experiment follows the same path as
in \cite{milstein2003simplest_dirichlet} (see also \cite{mil_tretyakov_book}). 
True solution is $\exp(-(1-t))(x_1^2 + x_2^2 )  + 10$. We evaluate the solution at  $(x_1, x_2) = (0,1)$, and therefore the   value is $10.367879 $ (6 dp). The "Error" column in Tables~\ref{sticky_euler_conv_normal_increments_table} and \ref{sticky_projected_euler_conv_normal_increments_table} as well as convergence plots in Figure~\ref{sticky_euler_conv_plot_1_normal_increments} confirm the results of Theorem~\ref{thrm_sticky_euler_conv} and Theorem~\ref{proj_conv_thrm}. 

\begin{table}[htbp]
\centering
\caption{Sticky Euler Scheme Convergence Results}
\label{sticky_euler_conv_normal_increments_table}
\begin{tabular}{c|c|c|c|c}
$h$ & $M$ &$\hat{\varphi}_{M}$ & Error & Avg Hit  \\
\hline
0.125 &  $5 \times 10^5$  & 10.649745 $\pm$ 0.00590 & 0.281866 & 1.92  \\
\hline
0.1 & $5 \times 10^5$ & 10.602825 $\pm$ 0.00550 & 0.234946 & 2.11  \\
\hline
0.0625 &  $5 \times 10^5$ & 10.525758 $\pm$ 0.00469 & 0.157878 & 2.56  \\
\hline
0.05 & $10^6$  & 10.495482 $\pm$ 0.003149  &  0.127603  & 2.81  \\
\hline
0.03125 & $10^6$ & 10.449102 $\pm$ 0.002912 & 0.081223  & 3.46  \\
\hline
0.025 & $10^6$ & 10.430940 $\pm$ 0.002831 & 0.063060   & 3.82  \\
\hline
0.0125 & $10^6$ & 10.397215 $\pm$ 0.002671  &  0.029335   & 5.22  \\
\end{tabular}
\end{table}

\begin{table}[htbp]
\centering
\caption{Projection Scheme Convergence Results}
\label{sticky_projected_euler_conv_normal_increments_table}
\begin{tabular}{c|c|c|c|c}
h & M &$\hat{\varphi}_{M}$ & Error & Avg Hit \\
\hline
0.125 & $5 \times 10^4$ & 6.699937 $\pm$ 0.01673 & 3.667943 & 2.28 \\
\hline
0.1 & $5 \times 10^4$ & 7.458848 $\pm$ 0.01280 & 2.909031 & 2.51 \\
\hline
0.0625 & $5 \times 10^4$ & 8.541244 $\pm$ 0.00899 & 1.826636 & 2.89 \\
\hline
0.05 & $5 \times 10^5$ & 8.843772 $\pm$ 0.00263 & 1.524107 & 3.07  \\
\hline
0.03125 & $5 \times 10^5$ & 9.248428 $\pm$ 0.00248 & 1.119451 & 3.44\\
\hline
0.025 & $5 \times 10^5$ & 9.369810 $\pm$ 0.00246 & 0.998069 & 3.60 \\
\hline
0.0125 & $5 \times 10^5$ & 9.611287 $\pm$ 0.00243 & 0.756592 & 4.06 \\
\end{tabular}
\end{table}
\vspace{-5pt}
\begin{figure}[htbp]
    \centering
    \includegraphics[width=0.6\textwidth]{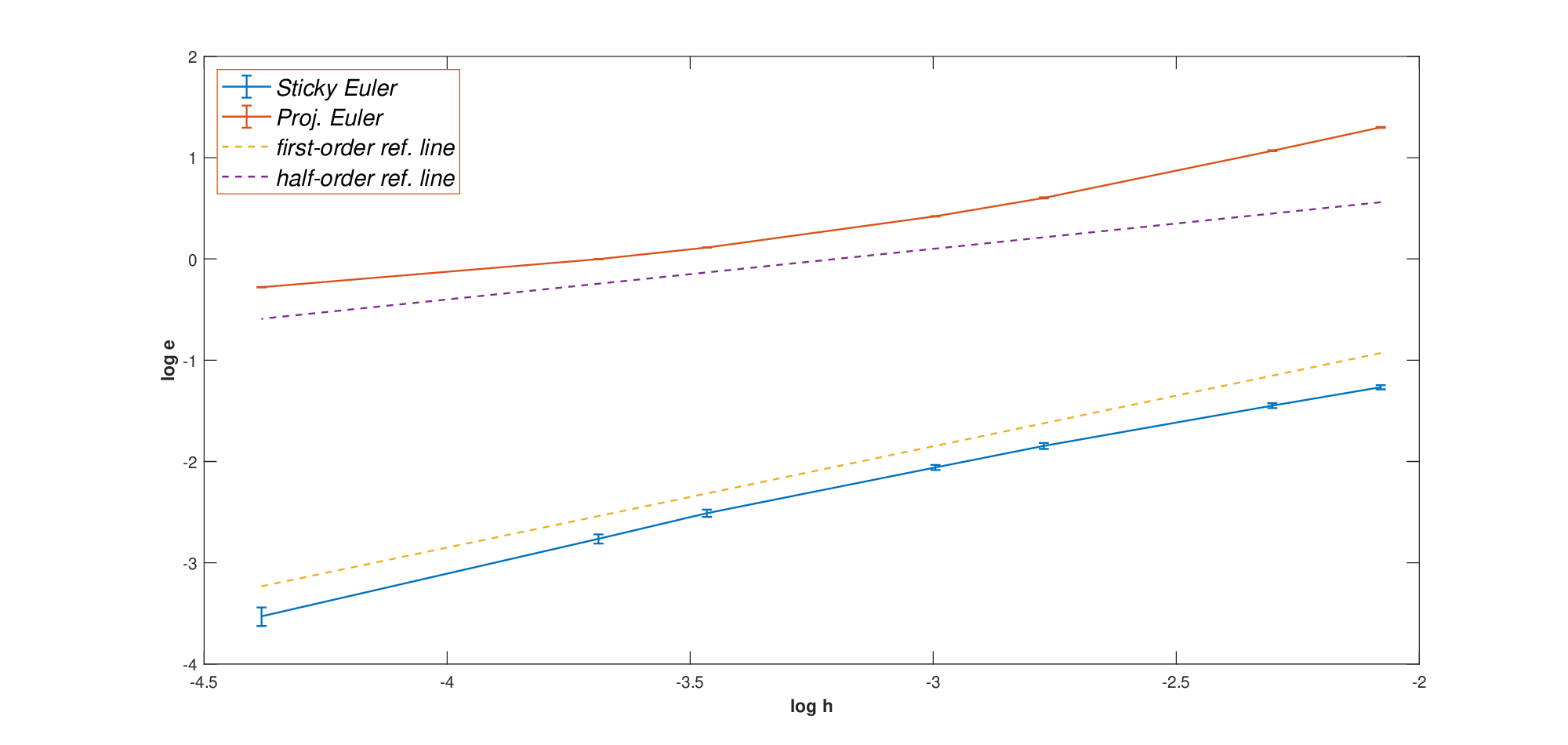}
    \caption{Plot confirming rate of convergence of Sticky Euler and Projected Euler schemes. }
    \label{sticky_euler_conv_plot_1_normal_increments}
\end{figure}

{
\small

\bibliography{references}
\bibliographystyle{abbrv}
}
\end{document}